\documentclass[bjps]{imsart}

\usepackage{amsfonts,amsmath,latexsym,amssymb,mathrsfs,amsthm}

\startlocaldefs

\def\numberlikeadb{\global\def\theequation{\thesection.\arabic{equation}}}
\numberlikeadb
\newtheorem{theorem}{Theorem}[section]
\newtheorem{lemma}[theorem]{Lemma}
\newtheorem{corollary}[theorem]{Corollary}

\newtheorem{proposition}[theorem]{Proposition}
\newtheorem{remark}[theorem]{Remark}

\numberwithin{equation}{section}

\endlocaldefs

\begin{document}

\begin{frontmatter}

% "Title of the Paper"
\title{Products of normal, beta and gamma random variables: Stein operators and distributional theory}

\runtitle{Products of normal, beta and gamma random variables}

\begin{aug}
% indicate corresponding author with \corref{}
 \author{\fnms{Robert E.} \snm{Gaunt}\thanksref{a,b}\corref{Robert E. Gaunt}\ead[label=e1]{robert.gaunt@manchester.ac.uk}\ead[label=e2,url]{www.foo.com}}
% \affiliation[a]{}

%\author{\fnms{Robert E} \snm{Gaunt}\thanksref{a,b}\ead[label=e1]{???}}
%\and
%\author{\fnms{???} \snm{???}\thanksref{b}\ead[label=e2]{???}}

\affiliation[a]{The University of Manchester}
\affiliation[b]{University of Oxford}

\address[a]{\printead{e1}}
%\address[b]{\printead{e2}}

\runauthor{R. E. Gaunt}

\end{aug}

\begin{abstract}In this paper, we extend Stein's method to products of independent beta, gamma, generalised gamma and mean zero normal random variables.  In particular, we obtain Stein operators for mixed products of these distributions, which include the classical beta, gamma and normal Stein operators as special cases.  These operators lead us to closed-form expressions involving the Meijer $G$-function for the probability density function and characteristic function of the mixed product of independent beta, gamma and central normal random variables.
\end{abstract}



\end{frontmatter}

\section{Introduction}
In 1972, Stein \cite{stein} introduced a powerful method for deriving bounds for normal approximation.  The method rests on the following characterisation of the normal distribution: $W\sim N(0,\sigma^2)$ if and only if 
\begin{equation} \label{stein lemma}\mathbb{E}[\mathcal{A}_Zf(W)]=0
\end{equation}
for all real-valued absolutely continuous functions $f$ such that $\mathbb{E}|f'(Z)|<\infty$ for $Z\sim N(0,\sigma^2)$, where the operator $\mathcal{A}_Z$, given by $\mathcal{A}_Zf(x)=\sigma^2f'(x)-xf(x)$, is often referred to as the Stein operator (see, for example, Ley et al$.$ \cite{ley}).  This gives rise to the following inhomogeneous differential equation, known as the Stein equation:
\begin{equation} \label{normal equation} \mathcal{A}_Zf(x)=h(x)-\mathbb{E}h(Z),
\end{equation}
where $Z \sim N(0,\sigma^2)$, and the test function $h$ is real-valued.   For any bounded test function, a solution $f_h$ to (\ref{normal equation}) exists (see Stein \cite{stein2}).  Now, evaluating both sides at any random variable $W$ and taking expectations gives
\begin{equation} \label{expect} \mathbb{E}[\mathcal{A}_Zf_h(W)]=\mathbb{E}h(W)-\mathbb{E}h(Z).
\end{equation}
Thus, the problem of bounding the quantity $\mathbb{E}h(W)-\mathbb{E}h(Z)$ reduces to solving (\ref{normal equation}) and bounding the left-hand side of (\ref{expect}).  This is of interest because there are a number of probability distances (for example, Kolmogorov and Wasserstein) of the form $d_{\mathcal{H}}(\mathcal{L}(W),\mathcal{L}(Z))=\sup_{h\in\mathcal{H}}|\mathbb{E}h(W)-\mathbb{E}h(Z)|$.  Hence
\begin{equation*}d_{\mathcal{H}}(\mathcal{L}(W),\mathcal{L}(Z))\leq \sup_{f\in\mathcal{F}(\mathcal{H})}|\mathbb{E}[\mathcal{A}_Zf(W)]|,
\end{equation*}
where $\mathcal{F}(\mathcal{H})=\{f_h\,:\,h\in\mathcal{H}\}$ is the collection of solutions to (\ref{normal equation}) for functions $h\in\mathcal{H}$.  
This basic approach applies equally well to non-normal limits, although different Stein operators are needed for different limit distributions.  For a nice account of the general approach, we refer the reader to Ley et al$.$ \cite{ley}.

Over the years, Stein's method has been adapted to many other distributions, such as the Poisson \cite{chen 0}, exponential \cite{chatterjee}, \cite{pekoz1}, gamma \cite{gaunt chi square} \cite{luk}, \cite{nourdin1} and beta \cite{dobler beta}, \cite{goldstein4}.  The first step in extending Stein's method to a new probability distribution is to obtain a Stein equation.  For the $\mathrm{Beta}(a,b)$ distribution with density $\frac{1}{B(a,b)}x^{a-1}(1-x)^{b-1}$, $0<x<1$, where $B(a,b)=\Gamma(a)\Gamma(b)/\Gamma(a+b)$ is the beta function, a Stein operator  commonly used in the literature (see \cite{dobler beta}, \cite{goldstein4} and \cite{schoutens})  is
\begin{equation}\label{beta equation}\mathcal{A}_{\mathrm{beta}}f(x)=x(1-x)f'(x)+(a-(a+b)x)f(x).
\end{equation}
For the $\Gamma(r,\lambda)$ distribution with density $\frac{\lambda^r}{\Gamma(r)}x^{r-1}\mathrm{e}^{-\lambda x}$, $x>0$, the Stein operator
\begin{equation}\label{gamma equation}\mathcal{A}_{\mathrm{gamma}}f(x)=xf'(x)+(r-\lambda x)f(x)
\end{equation}
is often used in the literature (see \cite{diaconis} and \cite{luk}).  In this paper, we extend Stein's method to products of independent beta, gamma, generalised gamma and central normal random variables.  In particular, we obtain natural generalisations of the operators (\ref{normal equation}), (\ref{beta equation}) and (\ref{gamma equation}) to products of such random variables.

\subsection{Products of independent normal, beta and gamma random variables}

The theory of products of independent random variables is far less well-developed than that for sums of independent random variables, despite appearing naturally in a various applications, such as the limits in a number of random graph and urn models (Hermann and Pfaffelhuber \cite{hp14} and Pek\"oz et al$.$ \cite{pekoz}).  However, fundamental methods for the derivation of the probability density function of products of independent random variables have been developed by Springer and Thompson \cite{springer66}.  Using the Mellin integral transform (as suggested by Epstein \cite{epstein}), the authors obtained explicit formulas for products of independent Cauchy and mean-zero normal variables, and some special cases of beta variables.  Building on this work, Springer and Thompson \cite{springer} showed that the p.d.f.s of the mixed product of mutually independent beta and gamma variables, and the products of independent central normal variables are Meijer $G$-functions (defined in Appendix B).  

The p.d.f$.$ of the product $Z=Z_1Z_2\cdots Z_N$ of
independent normal random variables $Z_i\sim N(0, \sigma_{i}^2)$, $i= 1,2,..., N$, is given by
\begin{equation}\label{MeijerN} p(x)=\frac{1}{(2\pi)^{N/2}\sigma}G_{0,N}^{N,0}\bigg(\frac{x^2}{2^N\sigma^2} \; \bigg| \;0\bigg), \quad x\in\mathbb{R},
\end{equation}
where $\sigma=\sigma_{1}\sigma_{2}\cdots\sigma_{N}$.  If $Z$ has density (\ref{MeijerN}), we say $Z$ has a \emph{product normal} distribution, and write $Z\sim \mathrm{PN}(N,\sigma^2)$.  The density of the product $X_1\cdots X_mY_1\cdots Y_n$, where $X_i\sim \mathrm{Beta}(a_i,b_i)$ and $Y_j\sim\Gamma(r_j,\lambda)$ and the $X_i$ and $Y_j$ are mutually independent, is, for $x>0$, given by
\begin{equation}\label{zxMeijerBC} p(x)=KG_{m,m+n}^{m+n,0}\bigg(\lambda^n x\; \bigg| \; \begin{matrix} a_1+b_1-1,\; a_2+b_2-1,\ldots, a_m+b_m-1 \\
a_1-1,\;a_2-1,\ldots,a_m-1,\;r_1-1,\ldots,r_n-1 \end{matrix} \bigg), 
\end{equation}
where
\[K=\lambda^n\prod_{i=1}^m\frac{\Gamma(a_i+b_i)}{\Gamma(a_i)}\prod_{j=1}^n\frac{1}{\Gamma(r_j)},\]
and we adopt the convention that the empty product is $1$.  A random variable with density (\ref{zxMeijerBC}) is said to have a \emph{product beta-gamma} distribution.  If (\ref{zxMeijerBC}) holds with $n=0$, the random variable is said to have a \emph{product beta} distribution, denoted by $\mathrm{PB}(a_1,b_1,\ldots,a_m,b_m)$; if (\ref{zxMeijerBC}) holds with $m=0$, then we call this a \emph{product gamma} distribution, denoted by $\mathrm{PG}(r_1,\ldots,r_m,\lambda)$.  We also say that a product of mutually independent beta, gamma and central normal random variables has a \emph{product beta-gamma-normal} distribution.  

For the product of two normals, (\ref{MeijerN}) simplifies to
\[p(x)=\frac{1}{\pi\sigma_{1}\sigma_{2}}K_0\bigg(\frac{|x|}{\sigma_{1}\sigma_{2}}\bigg), \quad x\in\mathbb{R},\]
where $K_0(x)$ is a modified Bessel function of the second kind (defined in Appendix B).  For the product of two gammas, (\ref{zxMeijerBC}) also simplifies (see Malik \cite{malik}):
\begin{equation*}p(x)=\frac{2\lambda^{r_1+r_2}}{\Gamma(r_1)\Gamma(r_2)}x^{(r_1+r_2)/2-1}K_{r_1-r_2}(2\lambda\sqrt{x}), \quad x>0.
\end{equation*}
Nadarajah and Kotz \cite{nk05} also give a formula, in terms of the Kummer function, for the density of the product of independent beta and gamma random variables.  However, in general, for 3 or more (mixed) products of independent beta, gamma and central normal random variables there are no such simplifications.

Pek\"oz et al$.$ \cite{pekoz3} extended Stein's method to generalised gamma random variables, denoted by $\mathrm{GG}(r,\lambda,q)$, having density
\begin{equation}p(x)=\frac{q\lambda^r}{\Gamma(\frac{r}{q})}x^{r-1}\mathrm{e}^{-(\lambda x)^q}, \quad x>0.
\end{equation} 
For $G\sim \mathrm{GG}(r,\lambda,q)$, we have $\mathbb{E}G^k=\lambda^{-q}\Gamma((r+k)/q)/\Gamma(r/q)$ and in particular $\mathbb{E}G^q=\frac{r}{q\lambda^q}$.    Special cases include $\mathrm{GG}(r,\lambda,1)=\Gamma(r,\lambda)$ and also $\mathrm{GG}(1,(\sqrt{2}\sigma)^{-1},2)=\mathrm{HN}(\sigma^2)$, where $\mathrm{HN}(\sigma^2)$ denotes a half-normal random variable: $|Z|$ where $Z\sim N(0,\sigma^2)$ (see D\"{o}bler \cite{dobler2} for Stein's method for the half normal distribution).  In this paper, we also extend Stein's method to the product of generalised gamma random variables, $\mathrm{GG}(r_i,\lambda,q)$, denoted by $\mathrm{PGG}(r_1,\ldots,r_n,\lambda,q)$.

\subsection{Product distribution Stein operators}

Recently, Gaunt \cite{gaunt pn} extended Stein's method to the product normal distribution, obtaining the following Stein operator for the $\mathrm{PN}(N,\sigma^2)$ distribution:
\begin{equation}\label{delta}\mathcal{A}_{Z}f(x)=\sigma^2A_Nf(x)-xf(x),
\end{equation}
where the operator $A_N$ is given by $A_Nf(x)=x^{-1}T^Nf(x)$ and $Tf(x)=xf'(x)$.  The Stein operator (\ref{delta}) is a $N$-th order differential operator that generalises the normal Stein operator (\ref{normal equation}) in a natural manner to products.  Such Stein operators are uncommon in the literature with the only other example being the $N$-th order operators of Goldstein and Reinert \cite{goldstein3}, involving orthogonal polynomials, for the normal distribution.  Very recently, Arras et al$.$ \cite{aaps16} have obtained a $N$-th order Stein operator for the distribution of a linear combination of $N$ independent gammas random variables, although their Fourier approach is very different to ours.  Also, in recent years, second order operators involving $f$, $f'$ and $f''$ have appeared in the literature for the Laplace \cite{pike}, variance-gamma distributions \cite{eichelsbacher}, \cite{gaunt vg}, generalized hyperbolic distributions \cite{gaunt gh} and the PRR family of \cite{pekoz}.

One of the main contributions of this paper is an extension of the product normal Stein operator (\ref{delta}) to mixed products of beta, gamma and normal random variables (see Propositions \ref{zxgammachara11}, \ref{prodbetaprop} and \ref{listprop}).  The Stein operators for these product distributions (given in Table 1) are higher order differential operators, which cannot be readily be obtained via standard methods, such as the generator method of Barbour \cite{barbour2} and G\"{o}tze \cite{gotze} and the density method of Stein et al$.$ \cite{stein3}.    

To obtain our product Stein equations, we use a conditioning argument to develop an algebra of Stein operators.  In our proofs, we shall use the differential operators  $T_rf(x)=xf'(x)+rf(x)$ and $B_{r_1,\ldots,r_n}f(x)=T_{r_n}\ldots T_{r_1}f(x)$ (note that $T_0\equiv T$).  It should be noted that whilst we restrict our attention to mixed products of betas, gammas and normals, we expect that the proofs techniques employed in this paper could also be applied to obtain Stein operators for independent products of a number of standard distributions.   

\begin{table}[ht]
\caption{Stein operators for product distributions.  $X\sim \mathrm{PB}(a_1,b_1\ldots,a_m,b_m)$, $Y\sim \mathrm{PG}(r_1,\ldots,r_n,\lambda)$ and $Z\sim \mathrm{PN}(N,\sigma^2)$ are mutually independent.} 
\centering
\begin{tabular}{c l l}
\hline 
Product $P$ & Stein operator $\mathcal{A}_Pf(x)$ & Order\\ [0.5ex]
\hline
$X$ & $B_{a_1,\ldots,a_m}f(x)-xB_{a_1+b_1,\ldots,a_m+b_m}f(x)$ & $m$ \\ [1.5ex]
$Y$ & $B_{r_1,\ldots,r_n}f(x)-\lambda^nx f(x)$ & $n$ \\ [1.5ex]
$Z$ & $\sigma^2A_Nf(x)-xf(x)$ &$N$ \\ [1.5ex]
$XY$ & $B_{a_1,\ldots,a_{m}}B_{r_1,\ldots,r_n}f(x)-\lambda^nxB_{a_1+b_1,\ldots,a_m+b_m}f(x)$ & $m+n$ \\ [1.5ex]
$XZ$ & $\sigma^2B_{a_1,\ldots,a_m}A_NB_{a_1,\ldots,a_m}f(x)$ &$2m+N$ \\ [1.5ex]
&$\quad-xB_{a_1+b_1,\ldots,a_m+b_m}B_{a_1+b_1-1,\ldots,a_m+b_m-1}f(x)$  \\ [1.5ex]
$YZ$ & $\sigma^2B_{r_1,\ldots,r_n}A_NB_{r_1,\ldots,r_n}f(x)-\lambda^{2n}xf(x)$ & $2n+N$ \\ [1.5ex]
$XYZ$ & $\sigma^2B_{a_1,\ldots,a_m}B_{r_1,\ldots,r_n}A_NB_{r_1,\ldots,r_n}B_{a_1,\ldots,a_m}f(x)$ & $2m+2n+N$ \\ [1.5ex]
& $\quad-\lambda^{2n}xB_{a_1+b_1,\ldots,a_m+b_m}B_{a_1+b_1-1,\ldots,a_m+b_m-1}f(x)$ \\  [1ex]

\hline
\end{tabular}
\label{table:nonlin}
\end{table}

It can be seen that the product beta and product gamma Stein operators reduce to the classical beta and gamma Stein operators when $m=1$ and $n=1$, respectively, as was so in the normal case.  In Section 2.2.2, we see that for certain parameter values the Stein operators for the products $XZ$ and $XYZ$ can be simplified to differential operators of lower order.  We give a precise criteria under which this occurs. 

In Proposition \ref{zxgammachara11}, we also obtain a operator for the generalised gamma distribution which leads to the following $\mathrm{PGG}(r_1,\ldots,r_n,\lambda,q)$ Stein operator:
\begin{equation}\label{gghhll}\mathcal{A}_{\mathrm{PGG}}f(x)=B_{r_1,\ldots,r_n}f(x)-(q\lambda^{q})^nx^q f(x).
\end{equation}
Taking $q=1$ in (\ref{gghhll}) yields the product gamma Stein operator $\mathcal{A}_Yf(x)$.  Taking $r_1=\cdots=r_N=1$, $\lambda=(\sqrt{2}\sigma)^{-1}$ and $q=2$ in (\ref{gghhll}) gives the following Stein operator for the product of $N$ independent half-normal random variables ($|Z|$ where $Z\sim\mathrm{PN}(N,\sigma^2)$):
\begin{equation*}\mathcal{A}_{\mathrm{PHN}}f(x)=\sigma^2T_1^Nf(x)-x^2 f(x),
\end{equation*}
where $x$ takes values in the interval $[0,\infty)$.  By allowing $x$ to takes values in $\mathbb{R}$, we obtain the following $\mathrm{PN}(N,\sigma^2)$ Stein operator
\begin{equation*}\tilde{\mathcal{A}}_Zf(x)=\sigma^2T_1^Nf(x)-x^2 f(x),
\end{equation*}
which differs from the $\mathrm{PN}(N,\sigma^2)$ operator (\ref{delta}).  Although, making the changes of variables $g(x)=xf(x)$ we have that $g'(x)=xf'(x)+f(x)$, and so
\begin{equation*}A_Ng(x)=x^{-1}T_0^Ng(x)=T_1^Nf(x),
\end{equation*}
from which we recover the Stein operator (\ref{delta}).

The product distribution Stein operators that are obtained in this paper have a number of interesting properties which are discussed in Remark \ref{nice eqn}.  However, despite their elegance, it is in general difficult to solve the corresponding Stein equation and bound the appropriate derivatives of the solution; further discussion is given in Remark \ref{rem222}.  

The classical normal, beta and gamma Stein equations are first order linear differential equations, and one can obtain uniform bounds for their solutions via elementary calculations.  Uniform bounds are available for the first four derivatives of the solution of the $\mathrm{PN}(2,\sigma^2)$ Stein equation (Gaunt \cite{gaunt pn}), and in Proposition \ref{appendixa2} we show that the $k$-th derivative of the solution of the $\mathrm{PG}(r_1,r_2,\lambda)$ Stein equation is uniformly bounded if the first $k$ derivatives of the test function $h$ are bounded.  Although, for all other cases of product distribution Stein equations we do not have bounds for derivatives of the solution.  

However, in Section 3, we consider a novel application of the product beta-gamma-normal Stein operator.  In Section 3.2, we use the operator to obtain a differential equation that the product beta-gamma-normal p.d.f$.$ must satisfy.  This allows us to `guess' a formula for the density function, which is then easily verified to be the correct formula via Mellin transforms.  This formula is new, and obtaining it directly using the inverse Mellin transform would have required some quite involved calculations.  From our formula we are able to obtain an expression for the characteristic function of the product normal-beta-gamma distribution, as well as estimates for the tail behaviour of the distribution.

\subsection{Outline of the paper}

We begin Section 2 by establishing some properties of the operators $A_N$ and $B_{r_1,\ldots,r_n}$.  We then obtain characterising equations for mixed products of beta, gamma and central normal random variables (Propositions \ref{zxgammachara11}, \ref{prodbetaprop} and \ref{listprop}), which lead to the operators of Table \ref{table:nonlin}.  In Section 2.2.2, we see that for certain parameter values simpler operators can be obtained.  In Section 2.3, we consider a Stein equation for the product of two independent gammas.  We solve the equation and show that the $k$-th derivative of the solution is uniformly bounded if the first $k$ derivatives of the test function $h$ are bounded.  

In Section 3, we obtain formulas for the p.d.f$.$ and characteristic function of the product beta-gamma-normal distribution, as well as an asymptotic formula for the tail behaviour of the distribution.  We use the product beta-gamma-normal Stein operator to propose a candidate formula for the p.d.f$.$ and then verify it using Mellin transforms.  

In Appendix A, we prove some results that were stated in the main text without proof.  Finally, Appendix B lists some basic properties of the Meijer $G$-function and modified Bessel functions that are used in this paper. 

\vspace{3mm}

\emph{Notation.} Throughout this paper, we shall let $T$ denote the operator $Tf(x)=xf'(x)$ and $A_N$ will denote the operator $A_Nf(x)=x^{-1}T^Nf(x)=\frac{\mathrm{d}}{\mathrm{d}x}(T^{N-1}f(x))$.  We also let $T_r$ denote the operator $T_rf(x)=xf'(x)+rf(x)$ and let $B_{r_1,\ldots,r_n}$ denote the operator $B_{r_1,\ldots,r_n}f(x)=T_{r_n}\cdots T_{r_1}f(x)$.  We shall let $C^n(I)$ be the space of functions on the interval $I$ with $n$ continuous derivatives, and $C_b^n(I)$ will denote the space of bounded functions on $I$ with $n$ continuous derivatives that are all bounded.   

\section{Stein operators for products of normal, beta and gamma random variables}

\subsection{Preliminary results}

We begin by presenting some useful properties of the operators $A_N$ and $B_{r_1,\ldots,r_n}$.  

\begin{lemma}\label{opablem}The operators $A_N$ and $B_{r_1,\ldots,r_n}$ have the following properties.

(i) The operators $T_r$ and $T_s$ are commutative, that is, $T_rT_sf(x)=T_sT_rf(x)$ for all $f\in C^2(\mathbb{R})$.  Thus, for all $f\in C^n(\mathbb{R})$, $B_{r_1,\ldots,r_n}f(x)=B_{r_{\sigma(1)},\ldots,r_{\sigma(n)}}f(x)$, where $\sigma$ is a permutation of the set $\{1,2,\ldots,n\}$.  

(ii) For all $f\in C^{n+N}(\mathbb{R})$, the operators $A_N$ and $B_{r_1,\ldots,r_n}$ satisfy
\begin{equation}\label{lefgh}A_NB_{r_1,\ldots,r_n}f(x)=B_{r_1+1,\ldots,r_n+1}A_Nf(x).
\end{equation} 
\end{lemma}

\begin{proof}(i) The first assertion follows since $T_rT_sf(x)=x^2f''(x)+(1+r+s)xf'(x)+rsf(x)=T_sT_rf(x)$, and the second assertion now follows immediately.  

(ii) As $A_1\equiv \frac{\mathrm{d}}{\mathrm{d}x}$, we have $A_1T_rf(x)=xf''(x)+(r+1)f'(x)=T_{r+1}A_1f(x)$.  Thus, on recalling that $A_Nf(x)=\frac{\mathrm{d}}{\mathrm{d}x}(T_0^{N-1}f(x))$ and using the fact that the operators $T_r$ and $T_s$ are commutative, we have
\begin{align*}A_NB_{r_1,\ldots,r_n}f(x)&=A_1T_0^{N-1}T_{r_1}\cdots T_{r_n}f(x)\\
&=A_1T_{r_1}\cdots T_{r_n}T_0^{N-1}f(x)\\
&=T_{r_1+1}A_1T_{r_2}\cdots T_{r_n}T_0^{N-1}f(x)\\
&=T_{r_1+1}\cdots T_{r_n+1}A_1T_0^{N-1}f(x)\\
&=B_{r_1+1,\ldots,r_n+1}A_Nf(x),
\end{align*}
where an iteration was applied to obtain the penultimate equality.
\end{proof}

The following fundamental formulas (Luke \cite{luke}, pp$.$ 24--26) disentangle the iterated operators $A_N$ and $B_{r_1,\ldots,r_n}$.  For $f\in C^n(\mathbb{R})$,
\begin{eqnarray}\label{nthoperator}A_Nf(x)&=&\sum_{k=1}^N{N\brace k}x^{k-1}f^{(k)}(x), \\
\label{zxople}B_{r_1,\ldots,r_n}f(x)&=&\sum_{k=0}^nc_{k,n}x^{k}f^{(k)}(x),
\end{eqnarray}
where ${n\brace k}=\frac{1}{k!}\sum_{j=0}^k(-1)^{k-j}\binom{k}{j}j^n$ are Stirling numbers of the second kind (Olver et al$.$ \cite{olver}, Chapter 26) and
\begin{equation}\label{cknfor}c_{k,m}=\frac{(-1)^k}{k!}\sum_{j=0}^k\frac{(-k)_j}{j!}(j+qa)\prod_{i=1}^{m-1}(j+qr_i),
\end{equation}
for $(a)_j=a(a+1)\cdots(a+j-1)$, $(a)_0=1$.

Applying (\ref{lefgh}) and (\ref{zxople}) gives that, for $f\in C^{m+n+N}(\mathbb{R})$,
\begin{align}B_{a_1,\ldots,a_m}A_NB_{b_1,\ldots,b_n}f(x)&=A_NB_{a_1-1,\ldots,a_m-1}B_{b_1,\ldots,b_n}f(x) \nonumber \\
&=x^{-1}T_0^NB_{a_1-1,\ldots,a_m-1}B_{b_1,\ldots,b_n}f(x) \nonumber \\
\label{abopabop}&=\sum_{k=1}^{m+n+N}\tilde{c}_{k,m+n+N}x^{k-1}f^{(k)}(x),
\end{align}
where the $\tilde{c}_{k,m+n+N}$ can be computed using (\ref{cknfor}).

\subsection{Stein operators}

With the preliminary results stated, we are now in a position to obtain Stein operators for mixed products of beta, gamma and central normal random variables, which give rise to the product distribution Stein operators of Table \ref{table:nonlin}.  From here on we shall suppose that the random variables $X\sim \mathrm{PB}(a_1,b_1\ldots,a_m,b_m)$, $Y\sim \mathrm{PG}(r_1,\ldots,r_n,\lambda)$ and $Z\sim \mathrm{PN}(N,\sigma^2)$ are mutually independent.  We shall also let $\mathcal{A}_Pf(x)$ be the operator for the product distribution $P$, as given in Table \ref{table:nonlin}.

\subsubsection{General parameters}

We firstly consider the case of mixed products of beta, gamma and central normal random variables with general parameter values.  In Section 2.2.2, we look at particular parameter values under which we can obtain some slightly simpler formulas for product distribution Stein operators.  We begin by recalling the product normal Stein operator that was obtained by Gaunt \cite{gaunt pn}.

\begin{proposition}\label{prodsteinlemma}Suppose $Z\sim\mathrm{PN}(N,\sigma^2)$.  Let $f\in C^N(\mathbb{R})$ be such that $\mathbb{E}|Zf(Z)|<\infty$ and $\mathbb{E}|Z^{k-1}f^{(k)}(Z)|<\infty$, $k=1,\ldots,N$.  Then
\begin{equation}\label{cracker} \mathbb{E}[\mathcal{A}_Zf(Z)]=0.
\end{equation}
\end{proposition}

We now present Stein operators for the product beta and product generalised gamma distributions; taking $q=1$ gives a product gamma distribution Stein operator.

\begin{proposition}\label{zxgammachara11}Suppose $G\sim\mathrm{PGG}(r_1,\ldots,r_n,\lambda,q)$.  Let $f\in C^n(\mathbb{R})$ be such that $\mathbb{E}|G^qf(G)|<\infty$ and $\mathbb{E}|G^{k}f^{(k)}(G)|<\infty$, $k=0,\ldots,n$, where $f^{(0)}\equiv f$.  Then
\begin{equation}\label{zxcracker11} \mathbb{E}[B_{r_1,\ldots,r_{n}}f(G)-(q\lambda^{q})^nG^qf(G)]=0.
\end{equation}
\end{proposition}

\begin{proof} We proceed by induction on $n$ and begin by proving the base case $n=1$.  The well-known characterisation of the gamma distribution, given in Luk \cite{luk}, states that if $U\sim \Gamma(r/q,\lambda)$, then
\begin{equation}\label{luckeqn}\mathbb{E}[Uf'(U)+(r/q-\lambda U)f(U)]=0
\end{equation}  
for all differentiable functions $f$ such that the expectation exists.  Now, if $V\sim \mathrm{GG}(r,\lambda,q)$, then $V\stackrel{\mathcal{D}}{=}(\lambda^{1-q}U)^{1/q}$.  Making the change of variables $V=(\lambda^{1-q}U)^{1/q}$ in (\ref{luckeqn}) leads to the following characterising equation for the $\mathrm{GG}(r,\lambda,q)$ distribution:
\begin{equation*}\mathbb{E}[Vf'(V)+(r-q\lambda^q V^q)f(V)]=0
\end{equation*} 
for all differentiable functions $f$ such that the expectation exists.  This can be written as $\mathbb{E}[T_rf(V)-q\lambda^qV^qf(V)]=0$, and so the result is true for $n=1$.

Let us now prove the inductive step.  We begin by defining $W_n=\prod_{i=1}^nV_i$  where $V_i\sim\mathrm{GG}(r_i,\lambda,q)$ and the $V_i$ are mutually independent.  We observe that $(T_pf)(a x)=T_pf_{a}(x)$ where $f_{a}(x)=f(a x)$, and so $(B_{p_1,\ldots p_l}f)(a x)=B_{p_1,\ldots p_l}f_{a}(x)$.  By induction assume that $(q\lambda^q)^{n}\mathbb{E}W_n^qg(W_n)=\mathbb{E}B_{r_1,\ldots,r_{n}}g(W_n)$ for all $g\in C^{n}(\mathbb{R})$ for some $n\geq 1$.  Then 
\begin{align*}&(q\lambda^q)^{n+1}\mathbb{E}W_{n+1}^qf(W_{n+1}) \\
&=(q\lambda^q)^{n+1}\mathbb{E}[V_{n+1}^q\mathbb{E}[W_n^qf_{V_{n+1}}(W_n)\mid V_{n+1}]]\\
&= q\lambda^q\mathbb{E}[V_{n+1}^q\mathbb{E}[B_{r_1,\ldots,r_{n}}f_{V_{n+1}}(W_n)\mid V_{n+1}]]\\
&= q\lambda^q\mathbb{E}[V_{n+1}^q(B_{r_1,\ldots,r_{n}}f)(W_nV_{n+1})]\\
&= q\lambda^q\mathbb{E}[\mathbb{E}[V_{n+1}^q (B_{r_1,\ldots,r_{n}}f_{W_n})(V_{n+1})\mid W_n]]\\
&= \mathbb{E}[\mathbb{E}[W_nV_{n+1}(B_{r_1,\ldots,r_{n}}f)'(W_nV_{n+1}) +r_{n+1}f(W_nV_{n+1})\mid W_n]]\\
&= \mathbb{E}B_{r_1,\ldots,r_{n+1}}f(W_{n+1}).
\end{align*} 
Thus, the result has been proved by induction on $n$.
\end{proof}

\begin{proposition}\label{prodbetaprop}Suppose $X\sim\mathrm{PB}(a_1,b_1,\ldots,a_m,b_m)$.  Let $f\in C^m((0,1))$ be such that $\mathbb{E}|X^kf^{(k)}(X)|<\infty$ and $\mathbb{E}|X^{k+1}f^{(k)}(X)|<\infty$, $k=0,\ldots,m$.  Then
\begin{equation}\label{zxcracker11beta} \mathbb{E}[\mathcal{A}_Xf(X)]=0.
\end{equation}
\end{proposition}

\begin{proof}The proof is similar to that of Proposition \ref{zxgammachara11}, and we proceed by induction on $m$.  Let $W_m=\prod_{i=1}^mX_i$  where $X_i\sim\mathrm{Beta}(a_i,b_i)$ and the $X_i$ are mutually independent.  The base case of the induction $m=1$ is the well-known characterisation (\ref{beta equation}) of the beta distribution.  By induction assume that $\mathbb{E}W_mB_{a_1+b_1,\ldots,a_m+b_m}g(W_m)=\mathbb{E}B_{a_1,\ldots,a_{m}}g(W_m)$ for all $g\in C^{m}(\mathbb{R})$ for some $m\geq 1$.  Then 
\begin{align*}&\mathbb{E}W_{m+1}B_{a_1+b_1,\ldots,a_{m+1}+b_{m+1}}f(W_{m+1}) \\
&=\mathbb{E}[X_{m+1}\mathbb{E}[W_mB_{a_1+b_1,\ldots,a_{m}+b_{m}}T_{a_{m+1}+b_{m+1}}f_{X_{m+1}}(W_{m})\mid X_{m+1}]]\\
&= \mathbb{E}[X_{m+1}\mathbb{E}[B_{a_1,\ldots,a_{m}}T_{a_{m+1}+b_{m+1}}f_{X_{m+1}}(W_{m})\mid X_{m+1}]]\\
&= \mathbb{E}[X_{m+1}(T_{a_{m+1}+b_{m+1}}B_{a_1,\ldots,a_{m}}f)(W_mX_{m+1})]\\
&= \mathbb{E}[\mathbb{E}[X_{m+1} (T_{a_{m+1}+b_{m+1}}B_{a_1,\ldots,a_{m}}f_{W_m})(X_{m+1})\mid W_m]]\\
&= \mathbb{E}[\mathbb{E}[X_{m+1}W_m(B_{a_1,\ldots,a_{m}}f_{W_m})'(X_{m+1}) +a_{m+1}f(W_mX_{m+1})\mid W_m]]\\
&= \mathbb{E}B_{a_1,\ldots,a_{m+1}}f(W_{m+1}),
\end{align*} 
and so necessity has been proved by induction on $m$.
\end{proof}

We now use the above product beta, gamma and normal Stein operators to obtain Stein operators for mixed products of such random variables.

\begin{proposition}\label{listprop}Let $X\sim \mathrm{PB}(a_1,b_1\ldots,a_m,b_m)$, $Y\sim \mathrm{PG}(r_1,\ldots,r_n,\lambda)$ and $Z\sim \mathrm{PN}(N,\sigma^2)$ be mutually independent. 

(i) Let $f\in C^{m+n}(\mathbb{R}_+)$ be such that $\mathbb{E}|(XY)^jf^{(j)}(XY)|<\infty$, $j=0,\ldots,m+n$, and $\mathbb{E}|(XY)^{k+1}f^{(k)}(XY)|<\infty$, $k=0,\ldots,m$.  Then
\begin{equation*}\mathbb{E}[\mathcal{A}_{XY}f(XY)]=0.
\end{equation*}

(ii) Let $f\in C^{2m+N}(\mathbb{R})$ be such that $\mathbb{E}|(XZ)^{j-1}f^{(j)}(XZ)|<\infty$, $j=1,\ldots,2m+N$, and $\mathbb{E}|(XZ)^{k+1}f^{(k)}(XZ)|<\infty$, $k=0,\ldots,2m$.   Then
\begin{equation}\label{sh1sh2}\mathbb{E}[\mathcal{A}_{XZ}f(XZ)]=0.
\end{equation}

(iii) Let $f\in C^{2n+N}(\mathbb{R})$ be such that $\mathbb{E}|YZf(YZ)|<\infty$ and additionally $\mathbb{E}|(YZ)^{k-1}f^{(k)}(YZ)|<\infty$, $k=1,\ldots,2n+N$.  Then
\begin{equation}\label{folpa}\mathbb{E}[\mathcal{A}_{YZ}f(YZ)]=0.
\end{equation}

(iv) Let $f\in C^{2m+2n+N}(\mathbb{R})$ be such that $\mathbb{E}|(XYZ)^{j-1}f^{(j)}(XYZ)|<\infty$, $j=1,\ldots,2m+2n+N$, and $\mathbb{E}|(XYZ)^{k+1}f^{(k)}(XYZ)|<\infty$, $k=0,\ldots,2m$.   Then
\begin{equation}\label{gamtimesnor}\mathbb{E}[\mathcal{A}_{XYZ}f(XYZ)]=0.
\end{equation}
\end{proposition}

\begin{proof}In our proof, we use the Stein operators of the product normal, product gamma and product beta distributions that were given in Propositions \ref{prodsteinlemma}, \ref{zxgammachara11} and \ref{prodbetaprop}, respectively.  We consider the four assertions separately.

(i) Recall that $(T_pf)(a x)=T_pf_{a}(x)$ where $f_{a}(x)=f(a x)$, and therefore $(B_{p_1,\ldots p_l}f)(a x)=B_{p_1,\ldots p_l}f_{a}(x)$.  From (\ref{zxcracker11beta}) and (\ref{zxcracker11}) (with $q=1$),we now have
\begin{align*}\lambda^n\mathbb{E}[XYB_{a_1+b_1,\ldots,a_m+b_m}f(XY)] &=\lambda^n\mathbb{E}[Y\mathbb{E}[XB_{a_1+b_1,\ldots,a_m+b_m}f_Y(X)\mid Y]] \\
&=\lambda^n\mathbb{E}[Y\mathbb{E}[B_{a_1,\ldots,a_m}f_Y(X)\mid Y]] \\
&=\lambda^n\mathbb{E}[YB_{a_1,\ldots,a_m}f(XY)] \\
&=\lambda^n\mathbb{E}[\mathbb{E}[YB_{a_1,\ldots,a_m}f_X(Y)\mid X]] \\
&=\mathbb{E}[\mathbb{E}[B_{r_1,\ldots,r_n}B_{a_1,\ldots,a_m}f_X(Y)\mid X]] \\
&=\mathbb{E}[B_{r_1,\ldots,r_n}B_{a_1,\ldots,a_m}f(XY)],
\end{align*}
as required.

(ii) We begin by noting that, since $A_Nf(x)=\frac{\mathrm{d}}{\mathrm{d}x}(T_0^{N-1}f(x))$, we have $(A_N)f(ax)=aA_Nf_a(x)$.  So from, (\ref{zxcracker11beta}) and (\ref{cracker}),
\begin{align*}&\mathbb{E}[XZB_{a_1+b_1,\ldots,a_m+b_m}B_{a_1+b_1-1,\ldots,a_m+b_m-1}f(XZ)]\\
&=\mathbb{E}[Z\mathbb{E}[XB_{a_1+b_1,\ldots,a_m+b_m}B_{a_1+b_1-1,\ldots,a_m+b_m-1}f_Z(X)\mid Z]] \\
&=\mathbb{E}[Z\mathbb{E}[B_{a_1,\ldots,a_m}B_{a_1+b_1-1,\ldots,a_m+b_m-1}f_Z(X)\mid Z]] \\
&=\mathbb{E}[\mathbb{E}[ZB_{a_1,\ldots,a_m}B_{a_1+b_1-1,\ldots,a_m+b_m-1}f_{X}(Z)\mid X]] \\
&=\sigma^2\mathbb{E}[\mathbb{E}[XA_NB_{a_1,\ldots,a_m}B_{a_1+b_1-1,\ldots,a_m+b_m-1}f_{X}(Z)\mid X]] \\
&=\sigma^2\mathbb{E}[XA_NB_{a_1,\ldots,a_m}B_{a_1+b_1-1,\ldots,a_m+b_m-1}f(XZ)].
\end{align*}
From Lemma \ref{opablem} we can obtain that $A_NB_{a_1,\ldots,a_m}B_{a_1+b_1-1,\ldots,a_m+b_m-1}=B_{a_1+b_1,\ldots,a_m+b_m}A_NB_{a_1,\ldots,a_m}$.  Applying this formula and (\ref{zxcracker11beta}) yields
\begin{align*}&\mathbb{E}[XZB_{a_1+b_1,\ldots,a_m+b_m}B_{a_1+b_1-1,\ldots,a_m+b_m-1}f(XZ)]\\
&=\sigma^2\mathbb{E}[XB_{a_1+b_1,\ldots,a_m+b_m}A_NB_{a_1,\ldots,a_m}f(XZ)] \\
&=\sigma^2\mathbb{E}[\mathbb{E}[XB_{a_1+b_1,\ldots,a_m+b_m}A_NB_{a_1,\ldots,a_m}f_Z(X)\mid Z]] \\
&=\sigma^2\mathbb{E}[\mathbb{E}[B_{a_1,\ldots,a_m}A_NB_{a_1,\ldots,a_m}f_Z(X)\mid Z]] \\
&=\sigma^2\mathbb{E}[B_{a_1,\ldots,a_m}A_NB_{a_1,\ldots,a_m}f(XZ)],
\end{align*}
as required.  

(iii) By a similar argument,
\begin{align*}\lambda^{2n}\mathbb{E}[YZf(YZ)]&=\lambda^{2n}\mathbb{E}[Z\mathbb{E}[Yf_Z(Y)\mid Z]] \\
&=\lambda^n\mathbb{E}[Z\mathbb{E}[B_{r_1,\ldots,r_n}f_Z(Y)\mid Z]] \\
&=\lambda^n\mathbb{E}[\mathbb{E}[ZB_{r_1,\ldots,r_n}f_Y(Z)\mid Y]] \\
&=\sigma^2\lambda^n\mathbb{E}[\mathbb{E}[YA_NB_{r_1,\ldots,r_n}f_Y(Z)\mid Y]] \\
&=\sigma^2\lambda^n\mathbb{E}[\mathbb{E}[YA_NB_{r_1,\ldots,r_n}f_Z(Y)\mid Z]] \\
&=\sigma^2\mathbb{E}[\mathbb{E}[B_{r_1,\ldots,r_n}A_NB_{r_1,\ldots,r_n}f_Z(Y)\mid Z]] \\
&=\sigma^2\mathbb{E}[B_{r_1,\ldots,r_n}A_NB_{r_1,\ldots,r_n}f(YZ)].
\end{align*}

(iv) Applying (\ref{zxcracker11beta}) and (\ref{gamtimesnor}) gives
\begin{align*}&\lambda^{2n}\mathbb{E}[XYZB_{a_1+b_1,\ldots,a_m+b_m}B_{a_1+b_1-1,\ldots,a_m+b_m-1}f(XYZ)] \\
&=\lambda^{2n}\mathbb{E}[YZ\mathbb{E}[XB_{a_1+b_1,\ldots,a_m+b_m}B_{a_1+b_1-1,\ldots,a_m+b_m-1}f_{YZ}(X)\mid YZ]] \\
&=\lambda^{2n}\mathbb{E}[YZ\mathbb{E}[B_{a_1,\ldots,a_m}B_{a_1+b_1-1,\ldots,a_m+b_m-1}f_{YZ}(X)\mid YZ]] \\
&=\lambda^{2n}\mathbb{E}[\mathbb{E}[YZB_{a_1,\ldots,a_m}B_{a_1+b_1-1,\ldots,a_m+b_m-1}f_{X}(YZ)\mid X]] \\
&=\sigma^2\mathbb{E}[\mathbb{E}[XB_{r_1,\ldots,r_n}A_NB_{r_1,\ldots,r_n}B_{a_1,\ldots,a_m}B_{a_1+b_1-1,\ldots,a_m+b_m-1}f_{X}(YZ)\mid X]] \\
&=\sigma^2\mathbb{E}[XB_{r_1,\ldots,r_n}A_NB_{r_1,\ldots,r_n}B_{a_1,\ldots,a_m}B_{a_1+b_1-1,\ldots,a_m+b_m-1}f(XYZ)].
\end{align*}
We now interchange the order of the operators using part (ii) of Lemma \ref{opablem} and then use (\ref{zxcracker11beta}) to obtain
\begin{align*}
&\lambda^{2n}\mathbb{E}[XYZB_{a_1+b_1,\ldots,a_m+b_m}B_{a_1+b_1-1,\ldots,a_m+b_m-1}f(XYZ)] \\
&=\sigma^2\mathbb{E}[XB_{a_1+b_1,\ldots,a_m+b_m}B_{r_1,\ldots,r_n}A_NB_{r_1,\ldots,r_n}B_{a_1,\ldots,a_m}f(XYZ)] \\
&=\sigma^2\mathbb{E}[\mathbb{E}[XB_{a_1+b_1,\ldots,a_m+b_m}B_{r_1,\ldots,r_n}A_NB_{r_1,\ldots,r_n}B_{a_1,\ldots,a_m}f_{YZ}(X)\mid YZ]] \\
&=\sigma^2\mathbb{E}[\mathbb{E}[B_{a_1,\ldots,a_m}B_{r_1,\ldots,r_n}A_NB_{r_1,\ldots,r_n}B_{a_1,\ldots,a_m}f_{YZ}(X)\mid YZ]] \\
&=\sigma^2\mathbb{E}[B_{a_1,\ldots,a_m}B_{r_1,\ldots,r_n}A_NB_{r_1,\ldots,r_n}B_{a_1,\ldots,a_m}f(XYZ)].
\end{align*}
This completes the proof.
\end{proof} 

\begin{remark}\label{nice eqn}We could have obtained first order Stein operators for the product normal, beta and gamma distributions using the density approach of Stein et al$.$ \cite{stein3} (see also Ley et al$.$ \cite{ley} for an extension of the scope of the density method).  However, this approach would lead to complicated operators involving Meijer $G$-functions.  We would expect that this would lead to various problems.  Firstly, bounding the derivatives of the solution could still be a challenging problem, and these derivatives might not even be bounded.  Moreover, the coupling techniques in the existing Stein's method literature seem to be most effective when the coefficients take a simple form.  Our Stein equations, on the other hand, are amenable to the use of couplings.  Indeed, Gaunt \cite{gaunt pn} used a generalised zero bias coupling in conjugation with the product normal Stein equation to prove product normal approximation theorems.   

From the formulas (\ref{nthoperator}) and (\ref{zxople}) for the operators $A_N$ and $B_{r_1,\ldots,r_n}$, it follows that the product Stein operators of Table \ref{table:nonlin} are linear ordinary differential operators with simple coefficients.  As an example, the Stein operator for the product $XYZ$ can be written as
\begin{equation*}\mathcal{A}_{XYZ}f(x)=\sigma^2\!\sum_{k=1}^{2m+2n+N}\!\alpha_{k,2m+2n+N}x^{k-1}f^{(k)}(x)-\lambda^{2n}\sum_{k=0}^{2m}\beta_{k,2m}x^{k+1}f^{(k)}(x),
\end{equation*}
where the $\alpha_{k,2m+2n+N}$ and $\beta_{k,2m}$ can be computed using (\ref{cknfor}).

As discussed in the Introduction, Stein operators of order greater than two are not common in the literature; however, our higher order product Stein operators seem to be natural generalisations of the classical normal, beta and gamma Stein operators to products.  It is interesting to note that whilst the product beta, gamma and normal Stein operators are order $m$, $n$ and $N$, respectively, the operator for their product is order $2m+2n+N$, whilst one might intuitively expect the order to be $m+n+N$.  The formula (\ref{ngbpdffor}) of Theorem \ref{ngbpdfthm} below for the p.d.f$.$ for the product $XYZ$ sheds light on this, and is discussed further in Remark \ref{rem:111}.  In Section 2.2.2, we shall see that for certain parameter values one can obtain lower order Stein operators for the product $XYZ$.  For example, the operator decreases by $m$ when $b_1=\cdots=b_m=1$, and this can also be understood from (\ref{ngbpdffor}) and properties of the Meijer $G$-function; this is also discussed in Remark \ref{rem:111}.  However, for general parameter values, we expect that a Stein operator for $XYZ$ with polynomial coefficients will be of order $2m+2n+N$; again, see Remark \ref{rem:111}.
\end{remark}

\subsubsection{Reduced order Stein operators}

By Lemma \ref{opablem}, we can write the Stein operators for the products $XZ$ and $XYZ$ as
\begin{align*}\mathcal{A}_{XZ}f(x)&=\sigma^2x^{-1}B_{a_1,\ldots,a_m}B_{a_1-1,\ldots,a_m-1}T_0^Nf(x)\\
&\quad-xB_{a_1+b_1,\ldots,a_m+b_m}B_{a_1+b_1-1,\ldots,a_m+b_m-1}f(x)
\end{align*}
and
\begin{align*}\mathcal{A}_{XYZ}f(x)&=\sigma^2x^{-1}B_{a_1,\ldots,a_m}B_{a_1-1,\ldots,a_m-1}B_{r_1,\ldots,r_n}B_{r_1-1,\ldots,r_n-1}T_0^Nf(x) \\
&\quad-\lambda^{2n}xB_{a_1+b_1,\ldots,a_m+b_m}B_{a_1+b_1-1,\ldots,a_m+b_m-1}f(x).
\end{align*}
With this representation, we can write down a simple criterion under which we can obtain Stein operators for the products $XZ$ and $XYZ$ with orders less than $2m+N$ and $2m+2n+N$ respectively.  For simplicity, we only consider the case of the product $XYZ$; we can treat the operator for product $XZ$ similarly.  

Define sets $R$ and $S$ by
\begin{align*}R&=\{a_1+b_1,\ldots,a_m+b_m,a_1+b_1-1,\ldots,a_m+b_m-1\}; \\
S&=\{a_1,\ldots,a_m,a_1-1,\ldots,a_m-1,r_1,\ldots,r_n,r_1-1,\ldots,r_n-1,0,\ldots,0\},
\end{align*}
where it is understood that there are $N$ zeros in $S$.  Then if $|R \cap S|=t$, the Stein operator $\mathcal{A}_{XYZ}f(x)$ can be reduced to one of order $2m+2n+N-t$.

To illustrate this criterion, we consider some particular parameter values.

(i) $b_1=\cdots=b_m=1$: $X$ is product of $m$ independent $U(0,1)$ random variables when also $a_1=\cdots=a_m=1$.  Here the Stein operator is
\begin{align}\label{chch11}\mathcal{A}_{XYZ}f(x)&=\sigma^2x^{-1}B_{a_1-1,\ldots,a_m-1}B_{r_1,\ldots,r_n}B_{r_1-1,\ldots,r_n-1}T_0^NB_{a_1,\ldots,a_m}f(x)\nonumber \\
&\quad-\lambda^{2n}xB_{a_1+1,\ldots,a_m+1}B_{a_1,\ldots,a_m}f(x),
\end{align}
where we used the fact that the operators $T_r$ and $T_s$ are commutative.  Taking $g(x)=B_{a_1,\ldots,a_m}f(x)$ then gives the $(m+2n+N)$-th order Stein operator
\begin{align}\label{chch22}\mathcal{A}g(x)&=\sigma^2x^{-1}B_{r_1,\ldots,r_n}B_{r_1-1,\ldots,r_n-1}T_0^NB_{a_1,\ldots,a_m}g(x)\nonumber\\
&\quad -\lambda^{2n}xB_{a_1+1,\ldots,a_m+1}g(x)\nonumber \\
&=\sigma^2B_{a_1,\ldots,a_m}B_{r_1,\ldots,r_n}A_NB_{r_1,\ldots,r_n}g(x)\nonumber \\
&\quad-\lambda^{2n}xB_{a_1+1,\ldots,a_m+1}g(x).
\end{align}
It should be noted that the Stein operator (\ref{chch22}) acts on a different class of functions to (\ref{chch11}).  To stress this point, we recall from part (iv) of Proposition \ref{listprop} that the operator (\ref{chch11}) acts on all functions $f\in C^{2m+2n+N}(\mathbb{R})$ such that $\mathbb{E}|(XYZ)^{j-1}f^{(j)}(XYZ)|<\infty$, $j=1,\ldots,2m+2n+N$, and $\mathbb{E}|(XYZ)^{k+1}f^{(k)}(XYZ)|<\infty$, $k=0,\ldots,2m$.  Whereas, (\ref{chch22}) acts on all functions $g\in C^{m+2n+N}(\mathbb{R})$ such that $\mathbb{E}|(XYZ)^{j-1}g^{(j)}(XYZ)|<\infty$, $j=1,\ldots,m+2n+N$, and $\mathbb{E}|(XYZ)^{k+1}g^{(k)}(XYZ)|<\infty$, $k=0,\ldots,m$.

In the subsequent examples, we shall not write down the resulting lower order Stein operators, although they can be obtained easily by similar calculations.

(ii) $a_1+b_1=\cdots=a_m+b_m=1$: $X$ is a product of $m$ independent arcsine random variables when also $a_1=\cdots=a_m=1/2$.  A Stein operator of order $m+2n+N$ can again be obtained. 

(iii) $m=n=N$, $a_1+b_1=\cdots=a_m+b_m=1$ and $r_1=\cdots=r_n=1$, so that $X$ and $Y$ are products of $m$ arcsine and $\mathrm{Exponential}(1)$ random variables respectively.  A Stein operator of order $3m$ can again be obtained.

(iv) $m=n=N$, $a_1+b_1=\cdots=a_m+b_m=1$ and $r_1=\cdots=r_n=2$. A Stein operator of order $3m$ can be obtained.

\subsection{A Stein equation for the product of two gammas}

In general, for the product distribution Stein equations that are obtained in this paper, it is difficult to solve the equation and bound the appropriate derivatives of the solution.  However, for the product normal Stein equation, Gaunt \cite{gaunt pn} obtained uniform bounds for the first four derivatives of the solution in the case $N=2$.  Here we show that, for the $\mathrm{PG}(r_1,r_2,\lambda)$ Stein equation, under certain conditions on the test function $h$, all derivatives of the solution are uniformly bounded. With a more detailed analysis than the one carried out in this paper we could obtain explicit constants; this is discussed in Remark \ref{rem333} below.  In Remark \ref{rem222} below, we discuss the difficulties of obtaining such estimates for more general product distribution Stein equations.

Taking $q=1$ in the characterisation of the product generalised gamma distribution given in Proposition \ref{zxgammachara11} leads to the following Stein equation for the $\mathrm{PG}(r_1,r_2,\lambda)$ distribution:
\begin{equation}\label{zxprod2gamma}x^2f''(x)+(1+r_1+r_2)xf'(x)+(r_1r_2-\lambda^2x)f(x)=h(x)-\mathrm{PG}_{r_1,r_2}^{\lambda}h,
\end{equation}
where $\mathrm{PG}_{r_1,r_2}^{\lambda}h$ denotes $\mathbb{E}h(Y)$, for $Y\sim \mathrm{PG}(r_1,r_2,\lambda)$.  The two functions $x^{-(r_1+r_2)/2} K_{r_1-r_2} (2\lambda\sqrt{x})$ and $x^{-(r_1+r_2)/2} I_{|r_1-r_2|} (2\lambda\sqrt{x})$ (the modified Bessel functions $I_\nu(x)$ and $K_\nu(x)$ are defined in Appendix B) form a fundamental system of solutions to the homogeneous equation (this can readily be seen from (\ref{realfeel})).  Therefore, we can use the method of variation of parameters (see Collins \cite{collins} for an account of the method) to solve (\ref{zxprod2gamma}).  The resulting solution is given in the following lemma and its derivatives are bounded in the next proposition.  The proofs are given in Appendix A.

\begin{lemma}\label{appendixa1}Suppose $h:\mathbb{R}_+ \rightarrow \mathbb{R}$ is bounded and let $\tilde{h}(x)=h(x)- \mathrm{PG}_{r_1,r_2}^{\lambda}h$.  Then the unique bounded solution $f:\mathbb{R}_+ \rightarrow \mathbb{R}$ to the Stein equation (\ref{zxprod2gamma}) is given by
\begin{align} \label{ink} f(x) &=-\frac{2K_{r_1-r_2}(2\lambda\sqrt{x})}{x^{(r_1+r_2)/2}} \int_0^xt^{(r_1+r_2)/2-1} I_{|r_1-r_2|}(2\lambda\sqrt{t}) \tilde{h}(t) \,\mathrm{d}t \nonumber \\
&\quad + \frac{2I_{|r_1-r_2|}(2\lambda\sqrt{x})}{x^{(r_1+r_2)/2}} \int_0^{x}  t^{(r_1+r_2)/2-1}K_{r_1-r_2}(2\lambda\sqrt{t})\tilde{h}(t)\,\mathrm{d}t\\
\label{pen}&=   -\frac{2K_{r_1-r_2}(2\lambda\sqrt{x})}{x^{(r_1+r_2)/2}} \int_0^xt^{(r_1+r_2)/2-1} I_{|r_1-r_2|}(2\lambda\sqrt{t}) \tilde{h}(t) \,\mathrm{d}t \nonumber \\
&\quad - \frac{2I_{|r_1-r_2|}(2\lambda\sqrt{x})}{x^{(r_1+r_2)/2}} \int_x^{\infty}  t^{(r_1+r_2)/2-1}K_{r_1-r_2}(2\lambda\sqrt{t})\tilde{h}(t)\,\mathrm{d}t.
\end{align}
\end{lemma}

\begin{proposition}\label{appendixa2}Suppose $h\in C_b^k(\mathbb{R}_+)$ and let $f$ denote the solution (\ref{ink}).  Then there exist non-negative constants $C_{0,k},C_{1,k},\ldots,C_{k,k}$ such that
\begin{equation}\label{appenbound}\|f\|\leq C_{0,0}\|\tilde{h}\|\quad\mbox{and}\quad\|f^{(k)}\|\leq C_{0,k}\|\tilde{h}\|+\sum_{j=1}^k C_{j,k}\|h^{(j)}\|, \quad k\geq1. 
\end{equation}
\end{proposition}

\begin{remark}\label{rem333} The solution $f$ can be bounded by
\begin{align*}|f(x)|&\leq 2\|\tilde{h}\|\frac{1}{x^{(r_1+r_2)/2}} \int_0^xt^{(r_1+r_2)/2-1}\big|K_{r_1-r_2}(2\lambda\sqrt{x}) I_{|r_1-r_2|}(2\lambda\sqrt{t}) \\
&\quad-
I_{|r_1-r_2|}(2\lambda\sqrt{x})K_{r_1-r_2}(2\lambda\sqrt{t})\big| \,\mathrm{d}t,  
\end{align*}
useful for `small' $x$, and
\begin{align*}|f(x)|&\leq 2\|\tilde{h}\|\frac{K_{r_1-r_2}(2\lambda\sqrt{x})}{x^{(r_1+r_2)/2}} \int_0^xt^{(r_1+r_2)/2-1} I_{|r_1-r_2|}(2\lambda\sqrt{t})  \,\mathrm{d}t \\
&\quad +2\|\tilde{h}\| \frac{I_{|r_1-r_2|}(2\lambda\sqrt{x})}{x^{(r_1+r_2)/2}} \int_x^{\infty}  t^{(r_1+r_2)/2-1}K_{r_1-r_2}(2\lambda\sqrt{t})\,\mathrm{d}t,
\end{align*}
useful for `large' $x$.  In the proof of Lemma \ref{appendixa1}, we use asymptotic formulas for modified Bessel functions to show that the above expressions involving modified Bessel functions are bounded for all $x>0$.  A more detailed analysis (see Gaunt \cite{gaunt} for an analysis that yields bounds for similar expressions involving integrals of modified Bessel functions) would allow one to obtain an explicit bound, uniform in $x$, for these quantities, which would yield an explicit value for the constant $C_{0,0}$.  By examining the proof of Proposition \ref{appendixa2}, we would then be able to determine explicit values for all $C_{j,k}$ by a straightforward induction.    However, since we do not use the product gamma Stein equation to prove any approximation results in this paper, we omit this analysis.
\end{remark}

\begin{remark}\label{rem222}For the $\mathrm{PN}(2,\sigma^2)$ and $\mathrm{PG}(r_1,r_2,\lambda)$ Stein equations, one can obtain a fundamental system of solutions to the homogeneous equation in terms of modified Bessel functions.  These functions are well-understood, meaning that the problem of bounding the derivatives of the solution is reasonably tractable.  However, for product distribution Stein equations in general, it is more challenging to bound the derivatives, because the Stein equation is of higher order and a fundamental system for the homogeneous equation is given in terms of less well-understood Meijer $G$-functions (this can be seen from (\ref{meidiffeqn})), which do not in general reduce to simpler functions.  See Gaunt \cite{gaunt pn}, Section 2.3.2 for a detailed discussion of this problem for the product normal case.  Obtaining bounds for other product distribution Stein equations is left as an interesting open problem, which if solved would mean that the Stein equations of this paper could be utilised to prove product, beta, gamma and normal approximation results.
\end{remark}

\section{Distributional properties of products of beta, gamma and normal random variables}

\subsection{Distributional theory}

Much of this section is devoted to proving Theorem \ref{ngbpdfthm} below which gives a formula for the p.d.f$.$ of the product beta-gamma-normal distribution.    Throughout this section we shall suppose that the random variables $X\sim \mathrm{PB}(a_1,b_1\ldots,a_m,b_m)$, $Y\sim \mathrm{PG}(r_1,\ldots,r_n,\lambda)$ and $Z\sim \mathrm{PN}(N,\sigma^2)$ are mutually independent, and denote their product by $W=XYZ$.   

\begin{theorem}\label{ngbpdfthm}The p.d.f$.$ of $W$ is given by
\begin{align}\label{ngbpdffor}p(x)=KG^{2m+2n+N,0}_{2m,2m+2n+N}\bigg(&\frac{\lambda^{2n}x^2}{2^{2n+N}\sigma^2}\; \bigg| \; \begin{matrix} \frac{a_1+b_1}{2},\ldots, \frac{a_m+b_m}{2}, \\
\frac{a_1}{2},\ldots,\frac{a_m}{2}, \frac{a_1-1}{2},\ldots,\frac{a_m-1}{2},\end{matrix}\cdots \nonumber \\
&\quad \cdots\begin{matrix}\frac{a_1+b_1-1}{2},\ldots, \frac{a_m+b_m-1}{2} \\
\frac{r_1}{2},\ldots,\frac{r_n}{2},\frac{r_1-1}{2},\ldots,\frac{r_n-1}{2},0,\ldots,0\end{matrix} \bigg),
\end{align}
where 
\begin{equation*}K=\frac{\lambda^n}{2^{2n+N/2}\pi^{(n+N)/2}\sigma}\prod_{i=1}^m\frac{\Gamma(a_i+b_i)}{2^{b_i}\Gamma(a_i)}\prod_{j=1}^n\frac{2^{r_j}}{\Gamma(r_j)}.
\end{equation*}
\end{theorem}

We prove this theorem in Section 3.3 by verifying that the Mellin transform of the product $XYZ$ is the same as the Mellin transform of the density (\ref{ngbpdffor}).  However, a constructive proof using the Mellin inversion formula would require more involved calculations.  In Section 3.2, we use the product beta-gamma-normal characterisation (Proposition \ref{listprop}, part (iv)) to motivate the formula (\ref{ngbpdffor}) as a candidate for the density of the product $W$.  As far as the author is aware, this is the first time a Stein characterisation has been applied to arrive at a new formula for the p.d.f$.$ of a distribution.

Before proving Theorem \ref{ngbpdfthm}, we note some simple consequences.  The product normal p.d.f$.$ (\ref{MeijerN}) is an obvious special case of the master formula (\ref{ngbpdffor}), and by using properties of the Meijer $G$-function one can also obtain the product beta-gamma density (\ref{zxMeijerBC}).  

\begin{remark}\label{rem:111}Let us now recall the sets $R$ and $S$ of Section 2.2.2:  
\begin{align*}R&=\{a_1+b_1,\ldots,a_m+b_m,a_1+b_1-1,\ldots,a_m+b_m-1\}; \\
S&=\{a_1,\ldots,a_m,a_1-1,\ldots,a_m-1,r_1,\ldots,r_n,r_1-1,\ldots,r_n-1,0,\ldots,0\},
\end{align*}
where there are $N$ zeros in set $S$.  By property (\ref{lukeformula}) of the Meijer $G$-function, it follows that the order of the $G$-function in the density (\ref{ngbpdffor}) decreases by $t$ if $|R\cap S|=t$.  This is precisely the same condition under which the order of the Stein operator $\mathcal{A}_{XYZ}f(x)$ decreases by $t$.  The reason for this becomes apparent in Section 3.2 when we note that the density (\ref{ngbpdffor}) satisfies the  differential equation $\mathcal{A}_{XYZ}^*p(x)=0$, where $\mathcal{A}_{XYZ}^*$ is an adjoint operator of $\mathcal{A}_{XYZ}$ with the same order.  Hence, the order of the Stein operator decreases precisely when the degree of the $G$-function in the density (\ref{ngbpdffor}) decreases.  

As an example of this simplification, taking $b_1=\cdots=b_m=1$ in (\ref{ngbpdffor}) and simplifying using (\ref{meijergidentity}), gives the following expression for the density:
\begin{align*}p(x)=\tilde{K}&G^{m+2n+N,0}_{m,m+2n+N}\bigg(\frac{\lambda^{2n}x^2}{2^{2n+N}\sigma^2}\; \bigg| \\
&\quad\quad \begin{matrix} \frac{a_1+1}{2},\ldots, \frac{a_m+1}{2} \\
\frac{a_1-1}{2},\ldots,\frac{a_m-1}{2},\;\frac{r_1}{2},\ldots,\frac{r_n}{2},\frac{r_1-1}{2},\ldots,\frac{r_n-1}{2},0,\ldots,0 \end{matrix} \bigg),
\end{align*}
where $\tilde{K}$ is the normalizing constant.  It is instructive to compare this with Example (i) of Section 2.2.2, in which a $(m+2n+N)$-th order Stein operator was obtained for this distribution.

The connection between the differential equation $\mathcal{A}_{XYZ}^*p(x)=0$ and the Stein operator $\mathcal{A}_{XYZ}f(x)$ also suggests that, for general parameter values, a Stein operator for $XYZ$ with polynomial coefficients will be a $(2m+2n+N)$-th order differential operator.  This is because the density (\ref{ngbpdffor}) is a Meijer $G$-function that, for general parameter values, satisfies the $(2m+2n+N)$-th order  $G$-function differential equation (\ref{meidiffeqn}).  We expect this to be the case unless the sets $R$ and $S$ share at least one element.
\end{remark}

Finally, we record two simple corollaries of Theorem \ref{ngbpdfthm}: a formula for the characteristic function of $W$ and tail estimates for its density.  We note that these formulas are new, but that there is quite an extensive literature on tail asymptotics for product distributions; see Hashorva and Pakes \cite{hp10} and references therein; see also Pitman and Racz \cite{pitman} for a recent neat derivation of the tail asymptotics for the density of the product of a beta random variable and an independent  gamma random variable.

\begin{corollary}\label{charfunctionqw}The characteristic function of $W$ is given by
\begin{align*}\phi(t)=MG^{2m+2n+N-1,1}_{2m+1,2m+2n+N-1}\bigg(&\frac{\lambda^{2n}}{2^{2n+N-2}\sigma^2t^2}\; \bigg| \; \begin{matrix} 1,\frac{a_1+b_1+1}{2},\ldots, \frac{a_m+b_m+1}{2}, \\
\frac{a_1+1}{2},\ldots,\frac{a_m+1}{2},\frac{a_1}{2},\ldots,\frac{a_m}{2}, \end{matrix}\cdots \\
&\quad\cdots\begin{matrix}
\frac{a_1+b_1}{2},\ldots, \frac{a_m+b_m}{2}\\
\frac{r_1+1}{2},\ldots,\frac{r_n+1}{2},\frac{r_1}{2},\ldots,\frac{r_n}{2},\frac{1}{2},\ldots,\frac{1}{2} \end{matrix} \bigg),
\end{align*}
where
\begin{equation*}M=\frac{1}{\pi^{(n+N-1)/2}}\prod_{i=1}^m\frac{\Gamma(a_i+b_i)}{2^{b_i}\Gamma(a_i)}\prod_{j=1}^n\frac{2^{r_j-1}}{\Gamma(r_j)}.
\end{equation*}
\end{corollary}

\begin{proof}Since the distribution of $W$ is symmetric about the origin, it follows that the characteristic function $\phi(t)$ is given by
\begin{equation*}\phi(t)=\mathbb{E}[\mathrm{e}^{itW}]=\mathbb{E}[\cos(tW)]=2\int_0^{\infty}\cos(tx)p(x)\,\mathrm{d}x.
\end{equation*}
Evaluating the integral using (\ref{meijergintegration}) gives
\begin{align*}\phi(t)=MG^{2m+2n+N,1}_{2m+2,2m+2n+N}\bigg(&\frac{\lambda^{2n}}{2^{2n+N-2}\sigma^2t^2}\; \bigg| \; \begin{matrix} \frac{1}{2},\frac{a_1+b_1}{2},\ldots, \frac{a_m+b_m}{2}, \\
\frac{a_1}{2},\ldots,\frac{a_m}{2}, \frac{a_1-1}{2},\ldots,\frac{a_m-1}{2},\end{matrix} \cdots \\
&\quad\cdots \begin{matrix} \frac{a_1+b_1-1}{2},\ldots, \frac{a_m+b_m-1}{2},0 \\
\frac{r_1}{2},\ldots,\frac{r_n}{2},\frac{r_1-1}{2},\ldots,\frac{r_n-1}{2},0,\ldots,0 \end{matrix} \bigg),
\end{align*}
where
\begin{align*}M=\frac{2K\sqrt{\pi}}{|t|}=\frac{1}{\pi^{(n+N-1)/2}}\frac{\Gamma(a_i+b_i)}{2^{b_i}\Gamma(a_i)}\prod_{j=1}^n\frac{2^{r_j-1}}{\Gamma(r_j)}\cdot\frac{\lambda^n}{2^{n+N/2-1}\sigma|t|},
\end{align*}
and simplifying the above expression using (\ref{meijergidentity}) and then (\ref{lukeformula}) completes the proof.
\end{proof}

\begin{corollary}The density (\ref{ngbpdffor}) of the random variable $W$ satisfies the asymptotic formula
\begin{align*}p(x)\sim N |x|^{\alpha}\exp\Bigg\{-(2n+N)\bigg(\frac{\lambda^{2n}x^2}{2^{2n+N}\sigma^2}\bigg)^{1/(2n+N)}\Bigg\}, \quad \text{as $|x|\rightarrow\infty$,}
\end{align*}
where
\begin{equation*}N=\frac{(2\pi)^{(2n+N-1)/2}}{(2n+N)^{1/2}}\bigg(\frac{\lambda^{2n}}{2^{2n+N}\sigma^2}\bigg)^{\alpha/2} K,
\end{equation*}
with $K$ defined as in Theorem \ref{ngbpdfthm}, and
\begin{equation*}\alpha=\frac{2}{2n+N}\bigg\{\frac{1-3n+N}{2}+\sum_{j=1}^nr_j-\sum_{j=1}^mb_j\bigg\}.
\end{equation*}
\end{corollary}

\begin{proof}Apply the asymptotic formula (\ref{asymg}) to the density (\ref{ngbpdffor}).
\end{proof}

\subsection{Discovery of Theorem \ref{ngbpdfthm} via the Stein characterisation}

Here we motivate the formula (\ref{ngbpdffor}) for the density $p$ of the product random variable $W$.  We do so by using the product beta-gamma-normal Stein characterisation to find a differential equation satisfied by $p$.  

By part (iv) of Proposition \ref{listprop} we have that
\begin{align}&\mathbb{E}[\sigma^2B_{a_1,\ldots,a_m}B_{r_1,\ldots,r_n}A_NB_{r_1,\ldots,r_n}B_{a_1,\ldots,a_m}f(W)\nonumber\\
\label{podfr}&\quad-\lambda^{2n}WB_{a_1+b_1,\ldots,a_m+b_m}B_{a_1+b_1-1,\ldots,a_m+b_m-1}f(W)]=0
\end{align}
for all $f\in C^{2m+2n+N}(\mathbb{R})$ such that $\mathbb{E}|W^{k-1}f^{(k)}(W)|<\infty$ for $1\leq k\leq 2m+2n+N$, and $\mathbb{E}|W^{k+1}f^{(k)}(W)|<\infty$ for $0\leq k\leq 2m$.  By using part (ii) of Lemma \ref{opablem} and that $A_Nf(x)=\frac{\mathrm{d}}{\mathrm{d}x}(T_0^{N-1}f(x))$ and $\frac{\mathrm{d}}{\mathrm{d}x}T_r^kf(x)=T_{r+1}^kf'(x)$, we can write
\begin{align*}A_NB_{r_1,\ldots,r_n}B_{a_1,\ldots,a_m}f(x)=B_{r_1+1,\ldots,r_n+1}B_{a_1+1,\ldots,a_m+1}T_1^{N-1}f'(x).
\end{align*}
On substituting into (\ref{podfr}), we see that the density $p(x)$ of $W$ satisfies the equation
\begin{align}&\int_{-\infty}^{\infty}\big\{\sigma^2B_{a_1,\ldots,a_m}B_{r_1,\ldots,r_n}B_{r_1+1,\ldots,r_n+1}B_{a_1+1,\ldots,a_m+1}T_1^{N-1}f'(x)\nonumber\\
\label{elsab}&\quad-\lambda^{2n}xB_{a_1+b_1,\ldots,a_m+b_m}B_{a_1+b_1-1,\ldots,a_m+b_m-1}f(x)\big\}p(x)\,\mathrm{d}x=0
\end{align}
for all functions $f$ in the class $\mathcal{C}_{p}$, which is defined by

(i) $f\in C^{2m+2n+N}(\mathbb{R})$;

(ii) $\mathbb{E}|W^{k-1}f^{(k)}(W)|<\infty$ for $1\leq k\leq 2m+2n+N$ and $\mathbb{E}|W^{k+1}f^{(k)}(W)|<\infty$ for $0\leq k\leq 2m$;

(iii) $x^{i+j+2}p^{(i)}(x)f^{(j)}(x)\rightarrow 0$ as $x\rightarrow\pm\infty$ for all $i,j$ such that $0\leq i+j\leq 2m-1$;

(iv) $x^{i+j}p^{(i)}(x)f^{(j)}(x)\rightarrow 0$ as $x\rightarrow\pm\infty$ for all $i,j$ such that $0\leq i+j\leq 2m+2n+N-1$.

It will later become apparent as to why it is helpful to have the additional conditions (iii) and (iv).  Note that $\mathcal{C}_{p}$ contains the set of all functions on $\mathbb{R}$ with compact support that are $2m+2n+N$ times differentiable.

We now note the following integration by parts formula.  Let $\gamma\in\mathbb{R}$ and suppose that $\phi$ and $\psi$ are differentiable.  Then
 \begin{align}&\int_{-\infty}^\infty  x^\gamma\phi(x)T_r\psi(x)\,\mathrm{d}x\nonumber\\
 &=\int_{-\infty}^{\infty}x^\gamma\phi(x)\{x\psi'(x)+r\psi(x)\}\,\mathrm{d}x \nonumber \\
 &=\int_{-\infty}^{\infty}x^{\gamma+1-r}\phi(x)\frac{\mathrm{d}}{\mathrm{d}x}(x^r\psi(x))\,\mathrm{d}x\nonumber \\
 &=\Big[x^{\gamma+1}\phi(x)\psi(x)\Big]_{-\infty}^{\infty}-\int_{-\infty}^{\infty}x^r\psi(x)\frac{\mathrm{d}}{\mathrm{d}x}(x^{\gamma+1-r}\phi(x))\,\mathrm{d}x\nonumber \\
 \label{skts}&=\Big[x^{\gamma+1}\phi(x)\psi(x)\Big]_{-\infty}^{\infty}-\int_{-\infty}^{\infty}x^\gamma\psi(x)T_{\gamma+1-r}\phi(x)\,\mathrm{d}x,
 \end{align}
 provided the integrals exist.  

We now return to equation (\ref{elsab}) and use (\ref{skts}) to obtain a differential equation that is satisfied by $p(x)$.  Using (\ref{skts}) we obtain
\begin{align*}&\int_{-\infty}^{\infty}xp(x)B_{a_1+b_1,\ldots,a_m+b_m}B_{a_1+b_1-1,\ldots,a_m+b_m-1}f(x)\,\mathrm{d}x\\
&=\Big[x^2p(x)B_{a_1+b_1,\ldots,a_{m-1}+b_{m-1}}B_{a_1+b_1-1,\ldots,a_m+b_m-1}f(x)\Big]_{-\infty}^{\infty}\\
&\quad-\int_{-\infty}^{\infty}xT_{2-a_m-b_m}p(x)B_{a_1+b_1-1,\ldots,a_m+b_m-1}f(x)\,\mathrm{d}x \\
&=-\int_{-\infty}^{\infty}xT_{2-a_m-b_m}p(x)B_{a_1+b_1-1,\ldots,a_m+b_m-1}f(x)\,\mathrm{d}x,
\end{align*}
where we used condition (iii) to obtain the last equality.  By a repeated application of integration by parts, using formula (\ref{skts}) and condition (iii), we arrive at
\begin{align*}&\int_{-\infty}^{\infty}xp(x)B_{a_1+b_1,\ldots,a_m+b_m}B_{a_1+b_1-1,\ldots,a_m+b_m-1}f(x)\,\mathrm{d}x\\
&=\int_{-\infty}^{\infty}xf(x)B_{3-a_1-b_1,\ldots,3-a_m-b_m}B_{2-a_1-b_1,\ldots,2-a_m-b_m}p(x)\,\mathrm{d}x.
\end{align*}
By a similar argument, using  (\ref{skts}) and condition (iv),  we obtain
\begin{align*}&\int_{-\infty}^{\infty}p(x)B_{a_1,\ldots,a_m}B_{r_1,\ldots,r_n}B_{r_1+1,\ldots,r_n+1}B_{a_1+1,\ldots,a_m+1}T_1^{N-1}f'(x)\,\mathrm{d}x \\
&=(-1)^N\int_{-\infty}^{\infty}f(x)\\
&\quad\times\frac{\mathrm{d}}{\mathrm{d}x}\big(T_0^{N-1}B_{-a_1,\ldots,-a_m}B_{-r_1,\ldots,-r_n}B_{1-r_1,\ldots,1-r_n}B_{1-a_1,\ldots,1-a_m}p(x)\big)\,\mathrm{d}x.
\end{align*}
Putting this together we have that
\begin{align*}&\int_{-\infty}^{\infty}\big\{(-1)^N\sigma^2x^{-1}T_0^{N}B_{-a_1,\ldots,-a_m}B_{-r_1,\ldots,-r_n}B_{1-r_1,\ldots,1-r_n}B_{1-a_1,\ldots,1-a_m}p(x)\\
&\quad-\lambda^{2n}xB_{3-a_1-b_1,\ldots,3-a_m-b_m}B_{2-a_1-b_1,\ldots,2-a_m-b_m}p(x)\big\}f(x)\,\mathrm{d}x=0
\end{align*}
for all $f\in\mathcal{C}_p$.  Since the integral in the above display is equal to zero for all $f\in\mathcal{C}_p$, it follows by a slight variation of the fundamental lemma of the calculus of variations (here we have restrictions on the growth of $f(x)$ in the limits $x\rightarrow\pm\infty$) that $p(x)$ satisfies the differential equation
\begin{align}&T_0^{N}B_{-a_1,\ldots,-a_m}B_{-r_1,\ldots,-r_n}B_{1-r_1,\ldots,1-r_n}B_{1-a_1,\ldots,1-a_m}p(x)\nonumber\\
\label{lolcofe}&-(-1)^N\sigma^{-2}\lambda^{2n}x^2B_{3-a_1-b_1,\ldots,3-a_m-b_m}B_{2-a_1-b_1,\ldots,2-a_m-b_m}p(x)=0.
\end{align}
We now make a change of variables to transform this differential equation to a Meijer $G$-function differential equation (see (\ref{meidiffeqn})).  To this end, let $y=\frac{\lambda^{2n}x^2}{2^{2n+N}\sigma^2}$.  Then, $x\frac{\mathrm{d}}{\mathrm{d}x}=2y\frac{\mathrm{d}}{\mathrm{d}y}$ and $p(y)$ satisfies the differential equation
\begin{align}&T_0^{N}B_{-\frac{a_1}{2},\ldots,-\frac{a_m}{2}}B_{-\frac{r_1}{2},\ldots,-\frac{r_n}{2}}B_{\frac{1-r_1}{2},\ldots,\frac{1-r_n}{2}}B_{\frac{1-a_1}{2},\ldots,\frac{1-a_m}{2}}p(y)\nonumber\\
\label{difnear}&\quad-(-1)^NyB_{\frac{3-a_1-b_1}{2},\ldots,\frac{3-a_m-b_m}{2}}B_{\frac{2-a_1-b_1}{2},\ldots,\frac{2-a_m-b_m}{2}}p(y)=0.
\end{align}
From (\ref{meidiffeqn}) it follows that a solution to (\ref{difnear}) is 
\begin{align*}p(y)=C&G^{2m+2n+N,0}_{2m,2m+2n+N}\bigg(y\; \bigg|\; \begin{matrix} \frac{a_1+b_1}{2},\ldots, \frac{a_m+b_m}{2},  \\
\frac{a_1}{2},\ldots,\frac{a_m}{2},  \end{matrix} \\
&\quad\quad\quad\quad\quad\begin{matrix}  \frac{a_1+b_1-1}{2},\ldots, \frac{a_m+b_m-1}{2} \\
 \frac{a_1-1}{2},\ldots,\frac{a_m-1}{2},\frac{r_1}{2},\ldots,\frac{r_n}{2},\frac{r_1-1}{2},\ldots,\frac{r_n-1}{2},0,\ldots,0 \end{matrix} \bigg),
\end{align*}
where $C$ is a constant.  Therefore, on changing variables, a solution to (\ref{lolcofe}) is given by
\begin{align*}p(x)=\widetilde{C}G^{2m+2n+N,0}_{2m,2m+2n+N}\bigg(&\frac{\lambda^{2n}x^2}{2^{2n+N}\sigma^2}\; \bigg| \; \begin{matrix} \frac{a_1+b_1}{2},\ldots, \frac{a_m+b_m}{2}, \\
\frac{a_1}{2},\ldots,\frac{a_m}{2}, \frac{a_1-1}{2},\ldots,\frac{a_m-1}{2},\end{matrix}\cdots \\
&\quad \cdots \begin{matrix} \frac{a_1+b_1-1}{2},\ldots, \frac{a_m+b_m-1}{2} \\
\frac{r_1}{2},\ldots,\frac{r_n}{2},\frac{r_1-1}{2},\ldots,\frac{r_n-1}{2},0,\ldots,0 \end{matrix} \bigg),
\end{align*}
where $\widetilde{C}$ is an arbitrary constant.  We can use the integration formula (\ref{meijergintegration2}) to determine a value of $\widetilde{C}$ such that $\int_{\mathbb{R}}p(x)\,\mathrm{d}x=1$.  With this choice of $\widetilde{C}$, $p(x)\geq 0$ and so $p$ is a density function.  However, there are $2m+2n+N$ linearly independent solutions to (\ref{lolcofe}) and whilst our solution $p$ is indeed a density function, a more detailed analysis would be required to rigorously prove that it is indeed the density function of the product beta-gamma-normal distribution.  Since a simple proof that $p$ is indeed the density function is now available to us via Mellin transforms, we decide to omit such an analysis.

\subsection{Proof of Theorem \ref{ngbpdfthm}}

Firstly, we define the Mellin transform and state some properties that will be useful to us.  The Mellin transform of a non-negative random variable $U$ with density $p$ is given by
\begin{equation*}M_U(s)=\mathbb{E}U^{s-1}=\int_0^{\infty}x^{s-1}p(x)\,\mathrm{d}x,
\end{equation*}
for all $s$ such that the expectation exists.  If the random variable $U$ has density $p$ that is symmetric about the origin then we can define the Mellin transform of $U$ by
\begin{equation*}M_U(s)=2\int_0^{\infty}x^{s-1}p(x)\,\mathrm{d}x.
\end{equation*}
The Mellin transform is useful in determining the distribution of products of independent random variables due to the property that if the random variables $U$ and $V$ are independent then 
\begin{equation}\label{melprod}M_{UV}(s)=M_U(s)M_V(s).
\end{equation}

\noindent \emph{Proof of Theorem \ref{ngbpdfthm}.}  It was shown by Springer and Thompson \cite{springer} that the Mellin transforms of $X$, $Y$ and $Z$ are
\begin{eqnarray*}M_X(s)&=& \prod_{j=1}^m\frac{\Gamma(a_j+b_j)}{\Gamma(a_j)}\frac{\Gamma(a_j-1+s)}{\Gamma(a_j+b_j-1+s)}, \\
M_Y(s)&=& \frac{1}{\lambda^{n(s-1)}}\prod_{j=1}^n\frac{\Gamma(r_j-1+s)}{\Gamma(r_j)}, \\
M_Z(s)&=&\frac{1}{\pi^{N/2}}2^{N(s-1)/2}\sigma^{s-1}\big[\Gamma\big(\tfrac{s}{2}\big)\big]^N.
\end{eqnarray*}
Then, as the random variables are independent, it follows from (\ref{melprod}) that
\begin{align}M_{XYZ}(s)&=\prod_{j=1}^m\frac{\Gamma(a_j+b_j)}{\Gamma(a_j)}\frac{\Gamma(a_j-1+s)}{\Gamma(a_j+b_j-1+s)}\times \frac{1}{\lambda^{n(s-1)}}\prod_{j=1}^n\frac{\Gamma(r_j-1+s)}{\Gamma(r_j)}\nonumber \\
\label{mel1}&\quad\times \frac{1}{\pi^{N/2}}2^{N(s-1)/2}\sigma^{s-1}\big[\Gamma\big(\tfrac{s}{2}\big)\big]^N.
\end{align}

Now, let $W$ be a random variable with density (\ref{ngbpdffor}).  Since the density of $W$ is symmetric about the origin, we have 
\begin{align}&M_W(s)=2\int_0^{\infty}x^{s-1}p(x)\,\mathrm{d}x\nonumber \\
&\quad=\frac{\lambda^n}{2^{2n+N/2}\pi^{(n+N)/2}\sigma}\prod_{j=1}^m\frac{\Gamma(a_j+b_j)}{2^{b_j}\Gamma(a_j)}\prod_{j=1}^n\frac{2^{r_j}}{\Gamma(r_j)} \times \bigg(\frac{2^{n+N/2}\sigma}{\lambda^n}\bigg)^s \times\big[\Gamma\big(\tfrac{s}{2}\big)\big]^N \nonumber\\
\label{mel2}&\quad\quad\times \prod_{j=1}^m\frac{\Gamma(\frac{a_j+s}{2})\Gamma\big(\frac{a_j-1+s}{2}\big)}{\Gamma\big(\frac{a_j+b_j+s}{2}\big)\Gamma\big(\frac{a_j+b_j-1+s}{2}\big)}\prod_{j=1}^n\Gamma\bigg(\frac{r_j+s}{2}\bigg)\Gamma\bigg(\frac{r_j-1+s}{2}\bigg),
\end{align}
where we used (\ref{meijergintegration2}) to compute the integral.  On applying the duplication formula $\Gamma(\frac{x}{2})\Gamma(\frac{x}{2}+\frac{1}{2})=2^{1-x}\sqrt{\pi}\Gamma(x)$ to (\ref{mel2}) we can deduce that the expressions (\ref{mel1}) and (\ref{mel2}) are equal.  Hence, the Mellin transforms of $W$ and $XYZ$ are equal and therefore $W$ and $XYZ$ are equal in distribution. \hfill $\square$

\appendix

\section{Further proofs}

\noindent\emph{Proof of Lemma \ref{appendixa1}.}  We begin by proving that there is at most one bounded solution to the $\mathrm{\mathrm{PG}}(r_1,r_2,\lambda)$ Stein equation (\ref{zxprod2gamma}).  Suppose $u$ and $v$ are bounded solutions to (\ref{zxprod2gamma}).  Define $w=u-v$.  Then $w$ is bounded and is a solution to the homogeneous equation
\begin{equation}\label{bcshcb}x^2w''(x)+(1+r_1+r_2)xw'(x)+(r_1r_2-\lambda^2x)w(x)=0.
\end{equation}
We now obtain the general solution to (\ref{bcshcb}).  We begin by noting that the general solution to the homogeneous equation
\[x^2s''(x)+xs'(x)-(x^2+(r_1-r_2)^2)s(x)=0\]
is given by $s(x)=CK_{r_1-r_2}(x)+DI_{|r_1-r_2|}(x)$ (see (\ref{realfeel})).  Here, we have used the fact that $K_{-\nu}(x)=K_{\nu}(x)$ for any $\nu\in\mathbb{R}$ and all $x>0$, which can be seen immediately from (\ref{kbesdef}).  Thus, $K_{|r_1-r_2|}(x)=K_{r_1-r_2}(x)$.  A simple change of variables now gives that $t(x)=CK_{r_1-r_2}(2\lambda\sqrt{x})+DI_{|r_1-r_2|}(2\lambda\sqrt{x})$ is the general solution to 
\begin{equation}\label{xbdsjc}x^2t''(x)+xt'(x)-(\lambda^2x+(r_1-r_2)^2/4)t(x)=0.
\end{equation}
Substituting $t(x)=x^{(r_1+r_2)/2}w(x)$ into (\ref{xbdsjc}) now shows that $w$ satisfies the differential equation (\ref{bcshcb}), and we have that the general solution to (\ref{bcshcb}) is given by
\[ w(x) = Aw_1(x)+Bw_2(x),\]
where 
\[w_1(x)=x^{-(r_1+r_2)/2} K_{r_1-r_2} (2\lambda\sqrt{x})\:\:\mbox{and}\:\: w_2(x)=x^{-(r_1+r_2)/2} I_{|r_1-r_2|} (2\lambda\sqrt{x}).\]
From the asymptotic formulas for modified Bessel functions (\ref{Ktend0}) and (\ref{roots}), it follows that in order to have a bounded solution we must take $A=B=0$, and thus $w=0$ and so there is at most one bounded solution to (\ref{zxprod2gamma}).

Since (\ref{zxprod2gamma}) is an inhomogeneous linear ordinary differential equation, we can use the method of variation of parameters (see Collins \cite{collins} for an account of the method) to write down the general solution of (\ref{zxprod2gamma}):
\begin{equation}\label{vargensoln}f(x)=-w_1(x)\int_a^x\frac{w_2(t)\tilde{h}(t)}{t^2W(t)}\,\mathrm{d}t+w_2(x)\int_b^x\frac{w_1(t)\tilde{h}(t)}{t^2W(t)}\,\mathrm{d}t,
\end{equation}
where $a$ and $b$ are arbitrary constants and $W(t)=W(w_1,w_2)=w_1w_2'-w_2w_1'$ is the Wronskian.  From the formula $W(K_\nu(x),I_\nu(x))=x^{-1}$ (Olver et al$.$ \cite{olver}, formula 10.28.2) and a simple computation we have $W(w_1(x),w_2(x))=\frac{1}{2}x^{-1-r_1-r_2}$.  Substituting the relevant quantities into (\ref{vargensoln}) and taking $a=b=0$ yields the solution (\ref{ink}).  That the solutions (\ref{ink}) and (\ref{pen}) are equal follows because $t^{(r_1-r_2)/2-1}K_{r_1-r_2}(2\lambda\sqrt{t})$ is proportional to the $\mathrm{PG}(r_1,r_2,\lambda)$ density function.

Finally, we show that the solution (\ref{ink}) is bounded if $h$ is bounded.  If $r_1\not=r_2$, then it follows from the asymptotic formulas for modified Bessel functions (see Appendix B.2.3) that the solution is bounded (here we check that the solution is bounded as $x\downarrow0$ using (\ref{ink}), and to verify that it is bounded as $x\rightarrow\infty$ we use (\ref{pen})).  If $r_1=r_2$, the same argument confirms that the solution is bounded as $x\rightarrow\infty$.  To deal with the limit $x\downarrow0$, we use the asymptotic formulas $I_0(x)\sim 1$ and $K_0(x)\sim-\log(x)$, as $x\downarrow0$, to obtain
\begin{align*}\lim_{x\downarrow0}|f(x)|&=\lim_{x\downarrow0}\frac{2}{x^{(r_1+r_2)/2}}\bigg|\int_0^xt^{(r_1+r_2)/2-1}\big[K_0(2\lambda\sqrt{x})I_0(2\lambda\sqrt{t})\\
&\quad-I_0(2\lambda\sqrt{x})K_0(2\lambda\sqrt{t})\big]\tilde{h}(t)\,\mathrm{d}t\bigg| \\
&=\lim_{x\downarrow0}\frac{1}{x^{(r_1+r_2)/2}}\int_0^xt^{(r_1+r_2)/2-1}\big[\log(x)-\log(t)\big]\tilde{h}(t)\,\mathrm{d}t \\
&\leq\|\tilde{h}\|\lim_{x\downarrow0}\frac{1}{x^{(r_1+r_2)/2}}\int_0^xt^{(r_1+r_2)/2-1}\big[\log(x)-\log(t)\big]\,\mathrm{d}t \\
&=\|\tilde{h}\|\lim_{x\downarrow0}\frac{1}{((r_1+r_2)/2)^2}=\frac{4\|\tilde{h}\|}{(r_1+r_2)^2}.
\end{align*}
Therefore the solution is bounded when $h$ is bounded.  This completes the proof.  \hfill $\square$

\vspace{3mm}

\noindent\emph{Proof of Proposition \ref{appendixa2}.}  In this proof, we make use of an iterative approach that first appeared D\"{o}bler \cite{dobler beta}, and was then developed further in D\"{o}bler et al$.$ \cite{dgv15}.  Denote the Stein operator for the $\mathrm{PG}(r_1,r_2,\lambda)$ distribution by $\mathcal{A}_{r_1,r_2,\lambda}f(x)$, so that the $\mathrm{PG}(r_1,r_2,\lambda)$ Stein equation is given by 
\begin{equation*}\mathcal{A}_{r_1,r_2,\lambda}f(x)=\tilde{h}(x).
\end{equation*}
Now, from the Stein equation (\ref{zxprod2gamma}) and a straightforward induction on $k$, we have that
\begin{align*}& x^2f^{(k+2)}(x)+(r_1+r_2+2k+1)xf^{(k+1)}(x)\\
&\quad+((r_1+k)(r_2+k)-\lambda^2x)f^{(k)}(x)=h^{(k)}(x)+k\lambda^2f^{(k-1)}(x),
\end{align*}   
which can be written as
\begin{equation*}\mathcal{A}_{r_1+k,r_2+k,\lambda}f^{(k)}(x)=h^{(k)}(x)+k\lambda^2 f^{(k-1)}(x).
\end{equation*}   
Now, by Lemma \ref{appendixa1}, there exists a constant $C_{r_1,r_2,\lambda}$ such that
\begin{equation*}\|f\|\leq C_{r_1,r_2,\lambda}\|\tilde{h}\|.
\end{equation*}
We also note that the test function $h'(x)+\lambda^2f(x)$ has mean zero with respect to the random variable $Y\sim \mathrm{PG}(r_1+1,r_2+1,\lambda)$, since
by the product gamma characterisation of Proposition \ref{zxgammachara11}, 
\begin{equation*}\mathbb{E}[h'(Y)+\lambda^2f(Y)]=\mathbb{E}[\mathcal{A}_{r_1+k,r_2+k,\lambda}f'(Y)]=0.
\end{equation*}
With these facts we therefore have that
\begin{align*}\|f'\|&\leq C_{r_1+1,r_2+1,\lambda}\|h'(x)+\lambda^2f(x)\| \leq C_{r_1+1,r_2+1,\lambda}\big(\|h'\|+\lambda^2\|f\|\big) \\
&\leq C_{r_1+1,r_2+1,\lambda}\big(\|h'\|+\lambda^2 C_{r_1,r_2,\lambda}\|\tilde{h}\|\big).
\end{align*}
Repeating this procedure then yields the bound (\ref{appenbound}), as required. \hfill $\square$

\section{Properties of the Meijer $G$-function and modified Bessel functions}

Here we define the Meijer $G$-function and modified Bessel functions and state some of their properties that are relevant to this paper.  For further properties of these functions see Luke \cite{luke} and Olver et al$.$ \cite{olver}.  

\subsection{The Meijer $G$-function}

\subsubsection{Definition}
The Meijer $G$-function is defined, for $z\in\mathbb{C}\setminus\{0\}$, by the contour integral:
\begin{align*}&G^{m,n}_{p,q}\bigg(z \; \bigg|\; {a_1,\ldots, a_p \atop b_1,\ldots,b_q} \bigg)\\
&\quad=\frac{1}{2\pi i}\int_{c-i\infty}^{c+i\infty}z^{-s}\frac{\prod_{j=1}^m\Gamma(s+b_j)\prod_{j=1}^n\Gamma(1-a_j-s)}{\prod_{j=n+1}^p\Gamma(s+a_j)\prod_{j=m+1}^q\Gamma(1-b_j-s)}\,\mathrm{d}s,
\end{align*}
where $c$ is a real constant defining a Bromwich path separating the poles of $\Gamma(s + b_j)$ from those of $\Gamma(1- a_j- s)$ and where we use the convention that the empty product is $1$.

\subsubsection{Basic properties}
The Meijer $G$-function is symmetric in the parameters $a_1,\ldots,a_n$; $a_{n+1},\ldots,a_p$; $b_1,\ldots,b_m$; and $b_{m+1},\ldots,b_q$.  Thus, if one the $a_j$'s, $j=n+1,\ldots,p$, is equal to one of the $b_k$'s, $k=1,\ldots,m$, the $G$-function reduces to one of lower order.  For example,
\begin{equation}\label{lukeformula}G_{p,q}^{m,n}\bigg(z \; \bigg| \;{a_1,\ldots,a_{p-1},b_1 \atop b_1,\ldots,b_q}\bigg)=G_{p-1,q-1}^{m-1,n}\bigg(z \; \bigg| \;{a_1,\ldots,a_{p-1} \atop b_2,\ldots,b_q}\bigg), \quad m,p,q\geq 1.
\end{equation}
The $G$-function satisfies the identity
\begin{equation}\label{meijergidentity}z^cG_{p,q}^{m,n}\bigg(z \; \bigg| \;{a_1,\ldots,a_p \atop b_1,\ldots,b_q}\bigg)=G_{p,q}^{m,n}\bigg(z \; \bigg| \;{a_1+c,\ldots,a_p+c \atop b_1+c,\ldots,b_q+c}\bigg).
\end{equation}

\subsubsection{Asymptotic expansion}For $x>0$,
\begin{equation}\label{asymg}G^{q,0}_{p,q}\bigg(x \; \bigg|\; {a_1,\ldots, a_p \atop b_1,\ldots,b_q} \bigg)\sim \frac{(2\pi)^{(\sigma-1)/2}}{\sigma^{1/2}}x^\theta \exp\big(-\sigma x^{1/\sigma}\big), \quad \text{as $x\rightarrow\infty$,}
\end{equation}
where $\sigma=q-p$ and
\begin{equation*}\theta=\frac{1}{\sigma}\bigg\{\frac{1-\sigma}{2}+\sum_{i=1}^qb_i-\sum_{i=1}^pa_i\bigg\}.
\end{equation*}

\subsubsection{Integration}
\begin{equation}\label{meijergintegration1}\int_0^{\infty}\mathrm{e}^{\omega x}G_{p,q}^{m,n}\bigg(\alpha x \; \bigg| \;{a_1,\ldots,a_p \atop b_1,\ldots,b_q}\bigg)\,\mathrm{d}x=\omega^{-1}G_{p+1,q}^{m,n+1}\bigg(\frac{\alpha}{\omega} \; \bigg| \;{0,a_1,\ldots,a_p \atop b_1,\ldots,b_q}\bigg).
\end{equation}
For the conditions under which this formula holds see Luke \cite{luke}, pp$.$ 166--167.

For $\alpha>0$, $\gamma>0$, $a_j<1$ for $j=1,\ldots,n$, and $b_j>-\frac{1}{2}$ for $j=1,\ldots,m$, we have
\begin{align}\label{meijergintegration}&\int_0^{\infty}\cos(\gamma x)G_{p,q}^{m,n}\bigg(\alpha x^2 \; \bigg| \;{a_1,\ldots,a_p \atop b_1,\ldots,b_q}\bigg)\,\mathrm{d}x\nonumber\\
&\quad=\sqrt{\pi}\gamma^{-1}G_{p+2,q}^{m,n+1}\bigg(\frac{4\alpha}{\gamma^2} \; \bigg| \;{\frac{1}{2},a_1,\ldots,a_p,0 \atop b_1,\ldots,b_q}\bigg).
\end{align}

The following formula follows from Luke \cite{luke}, formula (1) of section 5.6 and a change of variables:
\begin{align}\label{meijergintegration2}&\int_0^{\infty}x^{s-1}G_{p,q}^{m,n}\bigg(\alpha x^2 \; \bigg| \;{a_1,\ldots,a_p \atop b_1,\ldots,b_q}\bigg)\,\mathrm{d}x\\
&\quad=\frac{\alpha^{-s/2}}{2}\frac{\prod_{j=1}^m\Gamma(b_j+\frac{s}{2})\prod_{j=1}^n\Gamma(1-a_j-\frac{s}{2})}{\prod_{j=m+1}^q\Gamma(1-b_j-\frac{s}{2})\prod_{j=n+1}^p\Gamma(a_j+\frac{s}{2})}.
\end{align}
For the conditions under which this formula is valid see Luke, pp$.$ 158--159.  In particular, the formula is valid when $n=0$, $1\leq p+1\leq m\leq q$ and $\alpha>0$.

\subsubsection{Differential equation}
The $G$-function $f(z)=G^{m,n}_{p,q}\big(z\big|{a_1,\ldots,a_{p}, \atop b_1,\ldots,b_q}\big)$ satisfies the differential equation
\begin{equation}\label{meidiffeqn}(-1)^{p-m-n}zB_{1-a_1,\ldots,1-a_p}f(z)-B_{-b_1,\ldots,-b_q}f(z)=0,
\end{equation}
where $B_{r_1,\ldots,r_n}f(z)=T_{r_n}\cdots T_{r_1}f(z)$ for $T_rf(z)=zf'(z)+rf(z)$.

\subsection{Modified Bessel functions}

\subsubsection{Definitions}
The \emph{modified Bessel function of the first kind} of order $\nu \in \mathbb{R}$ is defined, for all $x\in\mathbb{R}$, by
\begin{equation*}\label{defI}I_{\nu} (x) = \sum_{k=0}^{\infty} \frac{1}{\Gamma(\nu +k+1) k!} \left( \frac{x}{2} \right)^{\nu +2k}.
\end{equation*}
The \emph{modified Bessel function of the second kind} of order $\nu\in\mathbb{R}$ is defined, for $x>0$, by
\begin{equation}\label{kbesdef}K_\nu(x)=\int_0^\infty \mathrm{e}^{-x\cosh(t)}\cosh(\nu t)\,\mathrm{d}t.
\end{equation}

\subsubsection{Representation in terms of the Meijer $G$-function}
\begin{eqnarray*}I_\nu(x)&=&i^{-\nu}G_{0,2}^{2,0}\bigg(-\frac{x^2}{4}\;\bigg|\;\frac{\nu}{2},-\frac{\nu}{2}\bigg), \quad x\in\mathbb{R}, \\
K_\nu(x)&=&\frac{1}{2}G_{0,2}^{2,0}\bigg(\frac{x^2}{4}\;\bigg|\;\frac{\nu}{2},-\frac{\nu}{2}\bigg), \quad x>0.
\end{eqnarray*} 

\subsubsection{Asymptotic expansions}
\begin{eqnarray}\label{Itend0}I_{\nu} (x) &\sim& \frac{1}{\Gamma(\nu +1)} \left(\frac{x}{2}\right)^{\nu}, \quad x \downarrow 0, \nonumber \\
\label{Ktend0}K_{\nu} (x) &\sim& \begin{cases} 2^{|\nu| -1} \Gamma (|\nu|) x^{-|\nu|}, &  x \downarrow 0, \: \nu \not= 0, \\
-\log x, &  x \downarrow 0, \: \nu = 0, \end{cases} \\
\label{roots} I_{\nu} (x) &\sim& \frac{\mathrm{e}^x}{\sqrt{2\pi x}}, \quad x \rightarrow \infty, \\
\label{Ktendinfinity} K_{\nu} (x) &\sim& \sqrt{\frac{\pi}{2x}} \mathrm{e}^{-x}, \quad x \rightarrow \infty. \nonumber
\end{eqnarray}

\subsubsection{Differential equation}
The modified Bessel differential equation is 
\begin{equation} \label{realfeel} x^2 f''(x) + xf'(x) - (x^2 +\nu^2)f(x) =0.\end{equation}
The general solution is $f(x)=AI_{\nu} (x) +BK_{\nu} (x).$

\section*{Acknowledgements} The author acknowledges support from EPSRC
grant EP/K032402/1 and is currently supported by a Dame Kathleen Ollerenshaw Research Fellowship.  The author would like to thank the anonymous referees for carefully reading this article and for their helpful comments and suggestions. The author would also like to thank Gesine Reinert for helpful discussions.

\end{document}